\documentclass[3p]{elsarticle}
\usepackage{bm}
\usepackage[parfill]{parskip}
\usepackage[utf8]{inputenc}
\usepackage{url}
\usepackage{amsmath,amssymb,amsfonts,amsthm}
\usepackage{color}
\usepackage{bm}
\usepackage{graphicx}
\usepackage{subcaption}
\usepackage{lipsum}
\usepackage{amsfonts}
\usepackage{graphicx}
\usepackage{epstopdf}
\usepackage{array}
\usepackage[hidelinks]{hyperref}
\usepackage{lineno}
\usepackage{bbm}
\usepackage{xfrac}
\usepackage{varwidth}
\usepackage[nice]{nicefrac}

\usepackage{amsmath}
\usepackage{algorithm}
\usepackage[noend]{algpseudocode}

\ifpdf
  \DeclareGraphicsExtensions{.eps,.pdf,.png,.jpg}
\else
  \DeclareGraphicsExtensions{.eps}
\fi

\newtheorem{theorem}{Theorem}
\newtheorem{lemma}{Lemma}
\newtheorem{definition}{Definition}

\makeatletter
\newsavebox\myboxA
\newsavebox\myboxB
\newlength\mylenA

\newcommand*\xoverline[2][0.75]{%
    \sbox{\myboxA}{$\m@th#2$}%
    \setbox\myboxB\null
    \ht\myboxB=\ht\myboxA%
    \dp\myboxB=\dp\myboxA%
    \wd\myboxB=#1\wd\myboxA
    \sbox\myboxB{$\m@th\overline{\copy\myboxB}$}
    \setlength\mylenA{\the\wd\myboxA}
    \addtolength\mylenA{-\the\wd\myboxB}%
    \ifdim\wd\myboxB<\wd\myboxA%
       \rlap{\hskip 0.5\mylenA\usebox\myboxB}{\usebox\myboxA}%
    \else
        \hskip -0.5\mylenA\rlap{\usebox\myboxA}{\hskip 0.5\mylenA\usebox\myboxB}%
    \fi}
\makeatother

\theoremstyle{definition}

\usepackage{tikz}
\usetikzlibrary{chains, positioning, arrows.meta, bending, shapes.arrows}
\usetikzlibrary{positioning,shapes,arrows,calc,shapes.geometric,backgrounds,fit}
\usepackage{pifont}
\usetikzlibrary{positioning,shapes,arrows,calc,shapes.geometric,backgrounds,fit}

\tikzstyle{block} = [rectangle,draw,minimum width=2em,align=center,rounded corners, minimum height=2em,scale=1.0]
\tikzstyle{blockleft} = [rectangle,draw,minimum width=2em,align=left,rounded corners, minimum height=2em,scale=1.0]
\tikzstyle{bigblock} = [rectangle,draw,minimum width=8em,align=center,rounded corners, minimum height=4em,scale=1.0]
\tikzstyle{connect} = [draw,-latex']
\tikzstyle{decision} = [diamond, draw, 
    text width=4.5em, text badly centered, node distance=3cm, inner sep=0pt]
\tikzstyle{line} = [draw, -latex']
\tikzstyle{cloud} = [draw, ellipse,fill=red!20, node distance=3cm,
    minimum height=2em]
\tikzstyle{linenoarrow}=[draw]

\makeatletter
\let\save@mathaccent\mathaccent
\newcommand*\if@single[3]{%
  \setbox0\hbox{${\mathaccent"0362{#1}}^H$}%
  \setbox2\hbox{${\mathaccent"0362{\kern0pt#1}}^H$}%
  \ifdim\ht0=\ht2 #3\else #2\fi
  }
\newcommand*\rel@kern[1]{\kern#1\dimexpr\macc@kerna}
\newcommand*\widebar[1]{\@ifnextchar^{{\wide@bar{#1}{0}}}{\wide@bar{#1}{1}}}
\newcommand*\wide@bar[2]{\if@single{#1}{\wide@bar@{#1}{#2}{1}}{\wide@bar@{#1}{#2}{2}}}
\newcommand*\wide@bar@[3]{%
  \begingroup
  \def\mathaccent##1##2{%
    \let\mathaccent\save@mathaccent
    \if#32 \let\macc@nucleus\first@char \fi
    \setbox\z@\hbox{$\macc@style{\macc@nucleus}_{}$}%
    \setbox\tw@\hbox{$\macc@style{\macc@nucleus}{}_{}$}%
    \dimen@\wd\tw@
    \advance\dimen@-\wd\z@
    \divide\dimen@ 3
    \@tempdima\wd\tw@
    \advance\@tempdima-\scriptspace
    \divide\@tempdima 10
    \advance\dimen@-\@tempdima
    \ifdim\dimen@>\z@ \dimen@0pt\fi
    \rel@kern{0.6}\kern-\dimen@
    \if#31
      \overline{\rel@kern{-0.6}\kern\dimen@\macc@nucleus\rel@kern{0.4}\kern\dimen@}%
      \advance\dimen@0.4\dimexpr\macc@kerna
      \let\final@kern#2%
      \ifdim\dimen@<\z@ \let\final@kern1\fi
      \if\final@kern1 \kern-\dimen@\fi
    \else
      \overline{\rel@kern{-0.6}\kern\dimen@#1}%
    \fi
  }%
  \macc@depth\@ne
  \let\math@bgroup\@empty \let\math@egroup\macc@set@skewchar
  \mathsurround\z@ \frozen@everymath{\mathgroup\macc@group\relax}%
  \macc@set@skewchar\relax
  \let\mathaccentV\macc@nested@a
  \if#31
    \macc@nested@a\relax111{#1}%
  \else
    \def\gobble@till@marker##1\endmarker{}%
    \futurelet\first@char\gobble@till@marker#1\endmarker
    \ifcat\noexpand\first@char A\else
      \def\first@char{}%
    \fi
    \macc@nested@a\relax111{\first@char}%
  \fi
  \endgroup
}
\makeatother

\journal{arXiv.org}

\newcommand{\TheTitle}{Energy stable and conservative dynamical low-rank approximation for the Su-Olson problem} 

\date{\today}

\usepackage{amsopn}
\DeclareMathOperator{\diag}{diag}

\begin{document}
\begin{frontmatter}

\title{\TheTitle}

\author[adressWuerzburg]{Lena Baumann}
\author[adressInnsbruck]{Lukas Einkemmer}
\author[adressWuerzburg]{Christian Klingenberg}
\author[adressAs]{Jonas Kusch}

\address[adressWuerzburg]{University of Wuerzburg, Department of Mathematics,  Wuerzburg, Germany, \href{mailto:lena.baumann@uni-wuerzburg.de}{lena.baumann@uni-wuerzburg.de} (Lena Baumann), \href{mailto:klingen@mathematik.uni-wuerzburg.de}{klingen@mathematik.uni-wuerzburg.de} (Christian Klingenberg) }
\address[adressInnsbruck]{University of Innsbruck, Numerical Analysis and Scientific Computing, Innsbruck, Austria, \href{mailto:lukas.einkemmer@uibk.ac.at}{lukas.einkemmer@uibk.ac.at}}
\address[adressAs]{Norwegian University of Life Sciences, Scientific Computing, \r{A}s, Norway, \href{mailto:jonas.kusch@nmbu.no}{jonas.kusch@nmbu.no}}


\begin{abstract}
Computational methods for thermal radiative transfer problems exhibit high computational costs and a prohibitive memory footprint when the spatial and directional domains are finely resolved. A strategy to reduce such computational costs is dynamical low-rank approximation (DLRA), which represents and evolves the solution on a low-rank manifold, thereby significantly decreasing computational and memory requirements. Efficient discretizations for the DLRA evolution equations need to be carefully constructed to guarantee stability while enabling mass conservation.
In this work, we focus on the Su-Olson closure leading to a linearized internal energy model and derive a stable discretization through an implicit coupling of internal energy and particle density. Moreover, we propose a rank-adaptive strategy to preserve local mass conservation. Numerical results are presented which showcase the accuracy and efficiency of the proposed low-rank method compared to the solution of the full system.
\end{abstract}


\end{frontmatter}

\section{Introduction}
\label{sec:intro}
Numerically solving the radiative transfer equations is a challenging task, especially due to the high dimensionality of the solution's phase space. A common strategy to tackle this issue is to choose coarse numerical discretizations and mitigate numerical artifacts \cite{lathrop1968ray,mathews1999propagation,morel2003analysis} which arise due to the insufficient resolution, see e.g.~\cite{Camminady2019RayEM,frank2020ray,abu2001angular,lathrop1971remedies,tencer2016ray}. Despite the success of these approaches in a large number of applications, the requirement of picking user-determined and problem dependent tuning parameters can render them impracticable. 
Another approach to deal with the problem's high dimensionality is the use of model order reduction techniques. A reduced order method which is gaining a considerable amount of attention in the field of radiation transport is dynamical low-rank approximation (DLRA) \cite{koch2007dynamical} due to its ability to yield accurate solutions while not requiring an expensive offline training phase. DLRA's core idea is to approximate the solution on a low-rank manifold and evolve it accordingly. Past work in the area of radiative transfer has focused on asymptotic-preserving schemes \cite{einkemmer2021asymptotic,einkemmer2022asymptotic}, mass conservation \cite{peng2021high}, stable discretizations \cite{kusch2023stability}, imposing boundary conditions \cite{kusch2021robust,hu2022adaptive} and implicit time discretizations \cite{peng2023sweep}. A discontinuous Galerkin discretization of the DLRA evolution equations for thermal radiative transfer has been proposed in \cite{marshakDLRA}.

A key building block of efficient, accurate and stable methods for DLRA is the construction of time integrators which are robust irrespective of small singular values in the solution \cite{kieri2016discretized}. Three integrators which move on the low-rank manifold while not being restricted by its curvature are the \emph{projector-splitting} (PS) integrator \cite{lubich2014projector}, the \emph{basis update \& Galerkin} (BUG) integrator \cite{ceruti2022unconventional}, and the \emph{parallel} integrator \cite{ceruti2023parallel}. Since the PS integrator evolves one of the required subflows backward in time, the BUG and parallel integrator are preferable for diffusive problems while facilitating the construction of stable numerical discretization for hyperbolic problems \cite{kusch2023stability}. Moreover, the BUG integrator allows for a basis augmentation step \cite{ceruti2022rank} which can be used to construct conservative schemes for the Schrödinger equation \cite{ceruti2022rank} and the Vlasov–Poisson equations \cite{einkemmer2022robust}.

In this work we consider the thermal radiative transfer equations using the Su-Olson closure. This leads to a linearized internal energy model for which we propose an energy stable and mass conservative DLRA scheme. The main novelties of this paper are:
\begin{itemize}
	\item \textit{A stable numerical scheme for thermal radiative transfer}: We show that a naive IMEX scheme fails to guarantee energy stability. To overcome this unphysical behaviour we propose a scheme which advances radiation and internal energy implicitly in a coupled fashion. In addition, our novel analysis gives a classic hyperbolic CFL condition that enables us to operate up to a time step size of $\Delta t = \text{CFL} \cdot \Delta x$.
    \item \textit{A mass conservative and rank-adaptive integrator}: We employ the basis augmentation step from \cite{ceruti2022rank} as well as an adaption of the conservative truncation strategy from \cite{einkemmer2022robust, guo2022} to guarantee local mass conservation and rank adaptivity. In contrast to \cite{einkemmer2022robust, guo2022} we do not need to impose conservation through a modified $\textit{L}$-step equation, but solely use the basis augmentation strategy from \cite{ceruti2022rank}.
\end{itemize}
Both these properties are extremely important as they ensure key physical principles and allow us to choose an optimal time step size which reduces the computational effort. Moreover,  we demonstrate numerical experiments which underline the derived stability and conservation properties of the proposed low-rank method while showing significantly reduced computational costs and memory requirements compared to the full-order system.

This paper is structured as follows: After the introduction in Section~\ref{sec:intro}, we review the background on thermal radiative transfer and dynamical low-rank approximation in Section~\ref{sec:background}. In Section~\ref{sec:DLRASuOlson} we present the evolution equations for the thermal radiative transfer equations when using the rank-adaptive BUG integrator. Section~\ref{sec:anglespace} discretizes the resulting equations in angle and space. The main method is presented in Section~\ref{sec:time} where a stable time discretization is proposed. We discuss local mass conservation of the scheme in Section~\ref{sec:mass}. Numerical experiments are demonstrated in Section~\ref{sec:num}.

\section{Background}\label{sec:background}
\subsection{Thermal radiative transfer}
In this work, we study radiation particles moving through and interacting with a background material. By absorbing particles, the material heats up and emits new particles which can in turn again interact with the background. This process is described by the thermal radiative transport equations
\begin{align*}
\frac{1}{c}\partial_t f(t,x,\mu) + \mu\partial_x f(t,x,\mu) &= \sigma(B(t,x)-f(t,x,\mu)),\\
\partial_t e(t,x) &= \sigma(\langle f(t,x,\cdot)\rangle_{\mu}-B(t,x)),
\end{align*}
where we omit boundary and initial conditions for now. This system can be solved for the particle density $f(t,x,\mu)$ and the internal energy $e(t,x)$ of the background medium. Here, $x\in D\subset \mathbb{R}$ is the spatial variable and $\mu\in [-1,1]$ denotes the directional (or velocity) variable. The opacity $\sigma$ encodes the rate at which particles are absorbed by the medium and we use brackets $\langle \cdot \rangle_{\mu}, \langle \cdot \rangle_x$ to indicate an integration over the directional domain and the spatial domain, respectively. Moreover, the speed of light is denoted by $c$ and the black body radiation at the material temperature $T$ is denoted by $B(T)$. It often is described by the Stefan-Boltzmann law 
\begin{align*}
    B(T) = acT^4,
\end{align*}
where $a = \frac{4\sigma_{\text{SB}}}{c}$ is the radiation density constant and $\sigma_{\text{SB}}$ the Stefan-Boltzmann constant. Different closures exist to determine a relation between the temperature $T$ and the internal energy $e$. Following the ideas of Pomraning \cite{pomraning1979} and Su and Olson \cite{suolson1997} we assume $e(T) = \alpha B(T)$. Without loss of generality we set $\alpha =1$ and obtain
\begin{subequations}\label{eq:thermalRadOrig}
\begin{align}
\label{eq1a}
\partial_t f(t,x,\mu) + \mu\partial_x f(t,x,\mu) &= \sigma(B(t,x)-f(t,x,\mu)),\\
\label{eq1b}
\partial_t B(t,x) &= \sigma(\langle f(t,x,\cdot)\rangle_{\mu}-B(t,x)).
\end{align}
\end{subequations}
We call this system the Su-Olson problem. It is a linear system for the particle density $f$ and the internal energy $B$ that is analytically solvable and and serves as a common benchmark for numerical considerations \cite{olson2000, mcclarren2008, mcclarren2008-2, mcclarren2008-3}. Note that we leave out the speed of light by doing a rescaling of time $\tau = t/c$ and in an abuse of notation use $t$ to denote $\tau$ in the remainder. Constructing numerical schemes to solve the above equation is challenging. First, the potentially stiff opacity term has to be treated by an implicit time integration scheme. Second, for three-dimensional spatial domains the computational costs and memory requirements of finely resolved spatial and angular discretizations become prohibitive. To tackle the high dimensionality, we choose a dynamical low-rank approximation which we introduce in the following.

\subsection{Dynamical low-rank approximation}\label{sec:recapDLRA}
The core idea of DLRA is to approximate the solution of a given equation $\partial_t f(t,x,\mu) = F(f(t,x,\mu))$ by a representation of the form
\begin{align}\label{eq:appox f DLRA}
    f(t,x,\mu) \approx \sum_{i,j=1}^r X_i(t,x)S_{ij}(t) V_{j}(t, \mu), 
\end{align}
where the orthonormal functions $\{X_i: i=1,...,r\}$ depend only on $t$ and $x$ and the orthonormal functions $\{V_j: j=1,...,r\}$ depend only on $t$ and $\mu$. The number of basis functions is set to $r$ and we call $r$ the rank of this approximation. This terminology stems from the matrix setting for which the concept of DLRA has been introduced \cite{koch2007dynamical}. Then, \eqref{eq:appox f DLRA} can be interpreted as a continuous analogue to the singular value decomposition for matrices. As representation \eqref{eq:appox f DLRA}
is not unique we impose the Gauge conditions $\langle \dot X_i, X_j\rangle_x = 0$ and $ \langle \dot V_i, V_j\rangle_{\mu} = 0$ from which we can conclude that $\{ X_i\}$ and $\{ V_j\}$ are uniquely determined for invertible $\mathbf{S} =(S_{ij})\in\mathbb{R}^{r\times r}$ \cite{koch2007dynamical, einkemmer2021asymptotic, einkemmerlubich2018}. That is, we seek for an approximation of $f$ that for each time $t$ lies in the manifold
\begin{align*}
    \mathcal{M}_r = \Bigg\{&f\in L^2(D\times [-1,1]) : f(\cdot, x,\mu) = \sum_{i,j=1}^r X_i(\cdot, x)S_{ij}(\cdot) V_{j}(\cdot, \mu) \text{ with invertible }\\
 &\mathbf{S} = (S_{ij})\in\mathbb{R}^{r\times r},
 X_i \in L^2(D), V_j\in L^2([-1,1]) \text{ and } \langle X_i, X_j\rangle_x = \delta_{ij}, \langle V_i, V_j\rangle_{\mu} = \delta_{ij} \Bigg\}.
\end{align*}
Note that in the following we denote the full rank and the low-rank solutions as $f$. Let $f(t,\cdot, \cdot)$ be a path on $\mathcal{M}_r$. A formal differentiation of $f$ with respect to $t$ leads to
\begin{align*}
    \dot f(t,\cdot, \cdot) = \sum_{i,j=1}^r \left( \dot X_i(t, \cdot)S_{ij}(t) V_{j}(t, \cdot) + X_i(t, \cdot) \dot S_{ij}(t) V_{j}(t, \cdot)+ X_i(t, \cdot) S_{ij}(t) \dot V_{j}(t, \cdot) \right).
\end{align*}
These functions restrict the solution dynamics onto the low-rank manifold $\mathcal{M}_r$ and constitute the corresponding tangent space which under the Gauge conditions reads
\begin{align*}
    \mathcal{T}_f\mathcal{M}_r = \Bigg\{&\dot f\in L^2(D\times [-1,1]): \dot f(\cdot, x,\mu) = \sum_{i,j=1}^r \Bigl(\dot X_i(\cdot, x)S_{ij}(\cdot) V_{j}(\cdot, \mu) + X_i(\cdot, x)\dot S_{ij}(\cdot) V_{j}(\cdot,\mu)\\
    &\ + X_i(\cdot, x)S_{ij}(\cdot) \dot V_{j}(\cdot, \mu) \Bigr)\\
    &\text{ with } \dot S_{ij} \in\mathbb{R},\dot X_i \in L^2(D), \dot V_j\in L^2([-1,1]) \text{ and } \langle \dot X_i, X_j\rangle_x = 0, \langle \dot V_i, V_j\rangle_{\mu} = 0 \Bigg\}.
\end{align*}
Having defined the low-rank manifold and its corresponding tangent space, we now wish to determine $f(t,\cdot,\cdot)\in\mathcal{M}_r$ such that $\partial_t f(t,\cdot,\cdot) \in\mathcal{T}_f\mathcal{M}_r$ and $\Vert \partial_t f(t,\cdot,\cdot) - F(f(t,\cdot,\cdot))\Vert_{L^2(D\times [-1,1])}$ is minimized. That is, one wishes to determine $f$ such that
\begin{align}\label{eq:variational}
    \langle \partial_t f(t,\cdot,\cdot) - F(f(t,\cdot,\cdot)), \dot f\, \rangle_{x,\mu} = 0 \quad \text{for all }\dot f\in\mathcal{T}_f\mathcal{M}_r.
\end{align}
The orthogonal projector onto the tangent plane $\mathcal{T}_f\mathcal{M}_r$ can be explicitly given as 
\begin{align*}
P(f)F(f) = \sum_{j=1}^r \langle V_j, F(f) \rangle_\mu V_j - \sum_{i,j=1}^r X_i \langle X_i V_j, F(f) \rangle_{x,\mu} V_j + \sum_{i=1}^r X_i \langle X_i, F(f) \rangle_x.
\end{align*}
With this definition at hand, we can reformulate \eqref{eq:variational} as
\begin{align*}
    \partial_t f(t,x,\mu) = P(f(t,x,\mu))F(f(t,x,\mu)).
\end{align*}
To evolve the approximation of the solution in time according to the above equation is not trivial. Indeed standard time integration schemes suffer from the curvature of the low-rank manifold, which is proportional to the smallest singular value of the low-rank solution \cite{koch2007dynamical}. Three integrators which move along the manifold without suffering from its high curvature exist: The projector--splitting integrator \cite{lubich2014projector}, the BUG integrator \cite{ceruti2022unconventional}, and the parallel integrator \cite{ceruti2023parallel}. In this work, we will use the basis-augmented extension to the BUG integrator \cite{ceruti2022rank} which we explain in the following.

The rank-adaptive BUG integrator \cite{ceruti2022rank} updates and augments the bases $\{ X_i\}, \{ V_j\} $ in parallel in the first two steps. In the third step, a Galerkin step is performed for the augmented bases followed by a truncation step to a new rank $r_1$. In detail, to evolve the approximation of the distribution function from $f(t_0,x,\mu) = \sum_{i,j=1}^r X_i^0(x) S_{ij}^0V_j^0(\mu)$ at time $t_0$ to $f(t_1,x,\mu) = \sum_{i,j=1}^{r_1} X_i^1(x) S_{ij}^1V_j^1(\mu)$ at time $t_1 = t_0 + \Delta t$ the integrator performs the following steps:

\textbf{\textit{K}-Step}: Write $K_j(t,x) = \sum_{i=1}^r X_i(t,x) S_{ij}(t)$. Then we obtain $f(t,x,\mu)=\sum_{j=1}^r K_j(t,x) V_j^0(\mu)$ with $\{V_j^0\}$ kept fixed in this step. The basis functions $X_i^0(x)$ with $i=1,...,r$ are updated by solving the partial differential equation
\begin{align*}
\partial_t K_j(t,x) = \left\langle V_j^0, F\left(\sum_{k=1}^r K_{k}(t,x) V_{k}^0\right) \right\rangle_\mu, \quad K_j(t_0,x) = \sum_{i=1}^r X_i^0(x) S_{ij}^0,
\end{align*}
and applying Gram Schmidt to $[K_j(t_1,x), X_i^0]=\sum_{i=1}^{2r} \widehat X_i^1(x) R_{ij}^1$. Then, the updated and augmented basis in physical space consists of $\widehat X_i^1(x)$ with $i=1,...,2r$. Note that $R_{ij}^1$ is discarded after this step. Compute $\widehat M_{ki} = \langle \widehat X_k^1, X_i^0 \rangle_x$.

\textbf{\textit{L}-Step}: Write $L_i(t,\mu) = \sum_{j=1}^r S_{ij}(t)V_j(t,\mu)$. Then we obtain $f(t,x,\mu)=\sum_{i=1}^r X_i^0 L_i(t,\mu)$ with $\{ X_i^0 \}$ kept fixed in this step. The basis functions $V_j^0(\mu)$ with $j=1,...,r$ are updated by solving the partial differential equation
\begin{align*}
\partial_t L_i(t,\mu) = \left\langle X_i^0, F\left(\sum_{\ell=1}^r X_{\ell}^0 L_{\ell}(t,\mu)\right) \right\rangle_x, \quad L_i(t_0, \mu) =\sum_{j=1}^r S_{ij}^0 V_j^0(\mu),
\end{align*}
and applying Gram Schmidt to $[L_i(t_1,\mu), V_j^0(\mu)]=\sum_{j=1}^{2r}  \widehat V_j^1(\mu
) R_{ij}^2$. Then, the updated and augmented basis in velocity space consists of $\widehat V_j^1(\mu)$ with $j=1,...,2r$. Note that $R_{ij}^2$ is discarded after this step. Compute $\widehat N_{\ell j} = \langle \widehat V_\ell^1, V_j^0\rangle_\mu$.

\textbf{\textit{S}-step}: Update $S_{ij}^0$ with $i,j=1,...,r$ to $\widehat S_{ij}^1$ with $i,j=1,...,2r$ by solving the ordinary differential equation
\begin{align*}
\dot{\widehat S}_{ij}(t) = \left\langle \widehat X_i^1 \widehat V_j^1, F\left(\sum_{\ell,k=1}^{2r} \widehat X_{\ell}^1 \widehat S_{\ell k}(t) \widehat V_k^1\right) \right\rangle_{x,\mu}, \quad \widehat S_{ij}(t_0) = \sum_{k, \ell = 1}^r \widehat M_{ik} S_{k\ell}^0 \widehat N_{j\ell}.
\end{align*}

\textbf{Truncation}: Let $\widehat S_{ij}^1$ be the entries of the matrix $\mathbf{\widehat S}^1$. Compute the singular value decomposition of $\mathbf{\widehat S}^1 = \mathbf{\widehat P \widehat \Sigma \widehat Q^\top}$ with $\mathbf{\Sigma} = \text{diag}(\sigma_j)$. Given a tolerance $\vartheta$, choose the new rank $r_1 \leq 2r$ as the minimal number such that 
\begin{align*}
\left(\sum_{j=r_1+1}^{2r} \sigma_j^2\right)^{1/2} \leq \vartheta.
\end{align*}
Let $\mathbf{S}^1$ with entries $S_{ij}^1$ be the $r_1 \times r_1$ diagonal matrix with the $r_1$ largest singular values and let $\mathbf{P}^1$ with entries $P_{ij}^1$ and $\mathbf{Q}^1$ with entries $Q_{ji}^1$ contain the first $r_1$ columns of $\mathbf{\widehat P}$ and $\mathbf{\widehat Q}$, respectively. Set $X_i^1(x) = \sum_{i=1}^{2r} \widehat X_i^1(x)P_{ij}^1$ for $i=1,...,r_1$ and $V_j^1(\mu) = \sum_{j=1}^{2r} \widehat V_j^1(\mu)Q_{ji}^1$ for $j=1,...,r_1$.

The updated approximation of the solution is then given by $f(t_1,x,\mu) = \sum_{i,j=1}^{r_1} X_i^1(x) S_{ij}^1 V_j^1(\mu)$ after one time step. Note that we are not limited to augmenting with the old basis, which we will use to construct our scheme.

\section{Dynamical low-rank approximation for Su-Olson}\label{sec:DLRASuOlson}
Let us now derive the evolution equations of the rank-adaptive BUG integrator for system \eqref{eq:thermalRadOrig}, i.e. the partial differential equations appearing in the \textit{K}- and \textit{L}-step and the ordinary differential equation for the $S$-step. To simplify notation, all derivations are performed for one spatial and one directional variable. However, the derivation trivially extends to higher dimensions. We start with considering the evolution equations for the low-rank approximation of the particle density \eqref{eq1a}. 

\textbf{$\textit{K}$-step:}
Write $K_j(t,x) = \sum_{i=1}^r X_i(t,x) S_{ij}(t)$. Then we have $f (t,x,\mu) = \sum_{j=1}^r K_j(t,x)V_{j}^0(\mu)$ for the low-rank approximation of the solution. Again $\{ V_j^0 \}$ denotes the set of orthonormal basis functions for the velocity space that shall be kept fixed in this step. Inserting this representation of $f$ into \eqref{eq1a} and projecting onto $V_k^0(\mu)$ gives the partial differential equation 
\begin{align}\label{3_Kstep continuous}
\partial_t K_k(t,x) = - \sum_{j=1}^r \partial_x K_j(t,x)\langle V^0_k, \mu V_j^0\rangle_{\mu} +  \sigma \left( B(t,x) \langle V^0_k\rangle_{\mu} - K_k(t,x)\right).
\end{align}

\textbf{$\textit{L}$-step:}
Write $L_i(t,\mu) =  \sum_{j=1}^r S_{ij}(t) V_j(t,\mu)$. Then we have $f(t,x,\mu) = \sum_{i=1}^r X_i^0(x) L_i (t,\mu)$ for the low-rank approximation of the solution. Again $\{X_i^0 \}$ denotes the set of spatial orthonormal basis functions that shall be kept fixed in this step. Inserting this representation of $f$ into \eqref{eq1a} and projecting onto $X_k^0(x)$ yields the partial differential equation 

\begin{align}\label{3_Lstep continous}
\partial_t L_k(t,\mu) = - \mu \sum_{i=1}^r \left\langle X_k^0, \frac{\mathrm{d}}{\mathrm{d}x} X_i^0\right\rangle_x L_i(t,\mu) + \sigma \left( \langle X_k^0, B(t,\cdot)\rangle_x - L_k(t,\mu)\right).
\end{align}
Lastly, we derive the augmented Galerkin step of the rank-adaptive BUG integrator. We denote the time updated spatial basis augmented with $X_i^0$ as $\widehat X_i^1$. The augmented directional basis $\widehat V_i^1$ is constructed in the corresponding way. Then, the augmented Galerkin step is constructed according to:

\textbf{$\textit{S}$-step:}
We use the initial condition $\widehat S_{ij}(t_0) = \sum_{\ell, k=1}^r \langle \widehat X_i^1  X_{\ell}^0 \rangle_x S_{\ell k}(t_0) \langle \widehat V_j^1  V_{k}^0\rangle_{\mu}$ and approximate the solution $f$ as $f(t,x,\mu) = \sum_{i,j = 1}^{2r} \widehat X_i^1(x) \widehat S_{ij}(t)\widehat V_j^1(\mu).$ Inserting this representation into \eqref{eq1a} and testing against $\widehat X_k^1$ and $\widehat V_\ell^1$ gives the ordinary differential equation
\begin{align}\label{3_Sstep continuous}
\dot{\widehat S}_{k\ell}(t)= &- \sum_{i,j=1}^{2r} \left\langle \widehat X_k^1, \frac{\mathrm{d}}{\mathrm{d}x} \widehat X_i^1 \right\rangle_x \widehat S_{ij}(t) \langle \widehat V_\ell^1, \mu \widehat V_j^1\rangle_\mu + \sigma \left(\langle \widehat X_k^1, B(t,\cdot) \rangle_x \langle \widehat V_\ell^1 \rangle_\mu - \widehat S_{k\ell}(t)\right)
\end{align}
from which we get the augmented quantity $\widehat S_{ij}(t)$. Inserting all augmented low-rank factors into \eqref{eq1b} leads to the partial differential equation
\begin{align}\label{3_ThermalRTE2 with DLRA}
\partial_t B(t,x) &= \sigma \left( \sum_{i,j=1}^{2r} \widehat X_i^1(x) \widehat S_{ij}(t)\langle \widehat V_j^1\rangle_{\mu}-B(t,x)\right).
\end{align}
Before repeating this process and evolving the subequations further in time we truncate back the augmented quantities to a new rank $r_1$ using a suitable truncation strategy.

\section{Angular and spatial discretization}\label{sec:anglespace}

Having derived the $\textit{K}$-, $\textit{L}$- and $\textit{S}$-step of the rank-adaptive BUG integrator, we can now proceed with discretizing in angle and space. For the angular discretization, we use the modal representations
\begin{align*}
V^0_j(\mu) \simeq \sum_{n = 0}^{N-1} V^0_{nj} P_n(\mu), \quad \widehat V^1_j(\mu) \simeq \sum_{n = 0}^{N-1} \widehat V^1_{nj} P_n(\mu), \quad L_i(t,\mu) \simeq \sum_{n = 0}^{N-1} L_{ni}(t) P_n(\mu),
\end{align*}
where $P_n$ are the normalized Legendre polynomials. Note that in the following, we use Einstein's sum convention when not stated otherwise to ensure compactness of notation. Let us define the matrix $\mathbf{A} \in \mathbb{R}^{N\times N}$ with entries $A_{mn} :=\langle P_m, \mu P_n \rangle_{\mu}$. Then we can rewrite $ \langle V^0_k, \mu V_j^0\rangle_{\mu} = V^0_{km}A_{mn}V^0_{jn}$. The evolution equations with angular discretization then read
\begin{subequations}
\label{3_evolution equations angular discretization} 
\begin{align}
\label{3_angular discretization Kstep}
\partial_t K_k(t,x) &= -\partial_x K_j(t,x) V_{nj}^0 A_{mn}  V_{mk}^0 + \sigma \left( B(t,x)  V_{0k}^0 - K_k(t,x)\right), \\
\label{3_angular discretization Lstep}
\dot L_{mk}(t) &= -  \left\langle X_k^0, \frac{\mathrm{d}}{\mathrm{d}x} X_i^0\right\rangle_x L_{ni} (t) A_{mn} + \sigma \left( \langle X_k^0, B(t,\cdot)\rangle_x \delta_{m0} - L_{mk}(t) \right), \\
\label{3_angular discretization Sstep}
\dot{\widehat S}_{k\ell}(t) &= - \left\langle \widehat X_k^1, \frac{\mathrm{d}}{\mathrm{d}x}  \widehat X_i^1 \right\rangle_x S_{ij}(t) \widehat V_{nj}^1 A_{mn}  \widehat V_{m\ell}^1 + \sigma \left( \langle \widehat X_k^1, B(t,\cdot) \rangle_x  \widehat V_{0 \ell}^1  - \widehat S_{k\ell}(t)\right).
\end{align}
For the angular discretization of \eqref{3_ThermalRTE2 with DLRA} we get
\begin{align}
\partial_t B(t,x) &= \sigma \left(\widehat X_i^1(x) \widehat S_{ij}(t) \widehat V_{0j}^1-B(t,x)\right) \label{3_angular discretization ThermalRTE2}.
\end{align}
\end{subequations}

To derive a spatial discretization we choose a spatial grid $x_1 < \cdots < x_{n_x}$ with equidistant spacing $\Delta x$. The solution in a given cell $p$ is then approximated by
\begin{align*}
&X_{pk}(t) \approx \frac{1}{\Delta x} \int_{x_p}^{x_{p+1}} X_k(t,x)\, \mathrm{d}x,
\quad K_{pk}(t) \approx \frac{1}{\Delta x} \int_{x_p}^{x_{p+1}} K_k(t,x)\, \mathrm{d}x,\\
&B_p(t) \approx \frac{1}{\Delta x} \int_{x_p}^{x_{p+1}} B(t,x)\, \mathrm{d}x\;.
\end{align*}
Spatial derivatives are approximated and stabilized through the tridiagonal stencil matrices $\mathbf D^x \approx \partial_x$ and $\mathbf D^{xx} \approx \frac{1}{2} \Delta x\partial_{xx}$ with entries
\begin{align*}
D_{p,p\pm 1}^{x}= \frac{\pm 1}{2\Delta x}\;,\qquad D_{p,p}^{xx}= -\frac{1}{\Delta x}\;, \quad D_{p,p\pm 1}^{xx}= \frac{1}{2\Delta x}\;.
\end{align*}
Applying the matrix $\mathbf D^x \in \mathbb{R}^{n_x \times n_x}$ corresponds to a first order and the stabilization matrix $\mathbf D^{xx} \in \mathbb{R}^{n_x \times n_x}$ to a second order central differencing scheme. Moreover, from now on we assume periodic boundary conditions. Recall the symmetric matrix $\mathbf{A}$. It is diagonalizable in the form $\mathbf{A} = \mathbf{Q}\mathbf{M}\mathbf{Q}^{\top}$ with $\mathbf{Q}$ orthogonal and $\mathbf{M} = \diag(\sigma_1,...,\sigma_n)$. We define matrix $|\mathbf{A}|$ as $|\mathbf{A}| = \mathbf{Q}|\mathbf{M}|\mathbf{Q}^{\top}$. We then obtain the spatially and angular discretized matrix ODEs
\begin{subequations}
\label{3_evolution equations spatial discretization} 
\begin{align}
\label{3_spatial discretization Kstep}
\dot{K}_{pk}(t) = &-D^x_{qp} K_{pj}(t)  V_{nj}^0 A_{mn}  V_{mk}^0 + D^{xx}_{qp} K_{pj}(t)  V_{nj}^0 |A|_{mn} V_{mk}^0\\ \nonumber
 &+ \sigma \left( B_p(t) V_{0k}^0 - K_{pk}(t)\right), \\
\label{3_spatial discretization Lstep}
\dot{L}_{mk}(t) = &- A_{mn}  L_{ni}(t) X_{pi}^0 D_{qp}^x X_{qk}^0 + |A|_{mn} L_{ni}(t) X_{pi}^0 D_{qp}^{xx} X_{qk}^0\\
\nonumber
&+ \sigma \left( \delta_{m0} B_p(t) X_{pk}^0 - L_{mk}(t)\right), \\
\label{3_spatial discretization Sstep}
\dot{\widehat S}_{k\ell}(t) = &- \widehat X_{pk}^1 D^x_{pq} \widehat X_{qi}^1 \widehat S_{ij}(t) \widehat V_{nj}^1 A_{mn} \widehat V_{m\ell}^1 + \widehat X_{pk}^1 D^{xx}_{pq} \widehat X_{qi}^1 \widehat S_{ij}(t) \widehat V_{nj}^1 |A|_{mn} \widehat V_{m\ell}^1\\ \nonumber
 &+ \sigma \left(\widehat X_{pk}^1 B_p(t) \widehat V_{0 \ell}^1  - \widehat S_{k\ell}(t)\right).
\end{align}
Lastly, we obtain from \eqref{3_angular discretization ThermalRTE2} for the internal energy $B$ the spatially discretized equation
\begin{align}\label{3_spatial discretization B}
\dot B_p(t) &= \sigma \left(\widehat X_{pi}^1 \widehat S_{ij}(t) \widehat V_{0j}^1-B_p(t)\right) = \sigma \left(u_{p0}^1(t)-B_p(t)\right),
\end{align}
\end{subequations}
where we use the notation $\widehat X_{pi}^1 \widehat S_{ij}(t) \widehat V_{mj}^1 =: u_{pm}^1(t)$. We can now show that the semi-discrete time-dependent system \eqref{3_evolution equations spatial discretization} is energy stable. For this, let us first give a definition of the total energy of the system:

\begin{definition}[Total energy]
    Let the matrix $\mathbf u^1(t) \in \mathbb{R}^{n_x \times N}$ with low-rank entries $u_{pm}^1(t) = \widehat X_{pi}^1 \widehat S_{ij}(t) \widehat V_{mj}^1$ denote the angularly and spatially discretized approximation of the solution of \eqref{eq1a} and $\mathbf{B}(t) \in \mathbb{R}^{n_x}$ be the spatially discretized approximation of the solution of \eqref{eq1b}. Then we call
   \begin{align*}
    E(t) := \frac{1}{2}\Vert \mathbf{u}^1(t) \Vert_F^2 + \frac{1}{2} \Vert \mathbf{B}(t) \Vert_E^2,
    \end{align*}
    with $\Vert \cdot \Vert_F$ denoting the Frobenius and $\Vert \cdot \Vert_E$ denoting the Euclidean norm,
    the \textit{total energy} of the system \eqref{3_evolution equations spatial discretization}.
\end{definition}
Further, we note the following properties of the chosen spatial stencil matrices which we write down denoting all sums explicitly:
\begin{lemma}[Summation by parts]
Let $y,z \in \mathbb{R}^{n_x}$ with indices $p,q=1,...,n_x$. In addition, we set $y_0 = y_{n_x}$ and $y_{n+1} = y_1$, for $z$ respectively, due to the periodic boundary conditions. Then the stencil matrices fulfill the following properties:
\begin{align*}
\sum_{p,q=1}^{n_x} y_p D_{pq}^x z_q = -\sum_{p,q=1}^{n_x} z_p D_{pq}^{x} y_q, \hspace{0.2cm} \sum_{p,q=1}^{n_x} z_p D_{pq}^{x} z_q = 0, \hspace{0.2cm} \sum_{p,q=1}^{n_x} y_p D_{pq}^{xx} z_q = \sum_{p,q=1}^{n_x} z_p D_{pq}^{xx} y_q.
\end{align*}
Moreover, let $\mathbf D^{+}\in\mathbb{R}^{n_x \times n_x}$ be defined as
\begin{align*}
D_{p,p}^{+}= \frac{- 1}{\sqrt{2\Delta x}}\;,\qquad D_{p,p + 1}^{+}= \frac{ 1}{\sqrt{2\Delta x}}\;.
\end{align*}
Then, $\sum_{p,q =1}^{n_x} z_p D_{pq}^{xx} z_q = - \sum_{p=1}^{n_x} \left(\sum_{q=1}^{n_x} D_{pq}^+ z_q\right)^2$.
\end{lemma}
\begin{proof}
The assertions follow directly by plugging in the definitions of the stencil matrices and rearranging the sums of the products in an adequate way:
\begin{align*}
\sum_{p,q=1}^{n_x} y_p D_{pq}^x z_q =& \ \frac{1}{2\Delta x} \sum_{p=1}^{n_x} y_p \left( z_{p+1} - z_{p-1} \right) = - \frac{1}{2\Delta x} \sum_{p=1}^{n_x} z_p \left( y_{p+1} - y_{p-1} \right)\\
=& \ -\sum_{p,q=1}^{n_x} z_p D_{pq}^x y_q, \\
\sum_{p,q=1}^{n_x} z_p D_{pq}^{x} z_q  =& - \sum_{p,q=1}^{n_x} z_p D_{pq}^{x} z_q = 0,
\end{align*}
\begin{align*}
\sum_{p,q=1}^{n_x} y_p D_{pq}^{xx} z_q =& \ - \frac{1}{\Delta x} \sum_{p=1}^{n_x} y_p z_p  + \frac{1}{2\Delta x} \sum_{p=1}^{n_x} y_p (z_{p+1} + z_{p-1})\\
=& \ - \frac{1}{\Delta x} \sum_{p=1}^{n_x} z_p y_p+ \frac{1}{2\Delta x} \sum_{p=1}^{n_x} z_p (y_{p+1} + y_{p-1})= \sum_{p,q=1}^{n_x} z_p D_{pq}^{xx} y_q,\\
\sum_{p,q=1}^{n_x} z_p D_{pq}^{xx} z_q =& \ - \frac{1}{\Delta x} \sum_{p=1}^{n_x} z_p^2 + \frac{1}{2\Delta x} \sum_{p=1}^{n_x} z_p (z_{p+1} + z_{p-1})\\
=& \ - \frac{1}{2\Delta x} \sum_{p=1}^{n_x} \left( z_p^2 - 2 z_p z_{p+1} + z_{p+1}^2\right)
= \ - \frac{1}{2\Delta x} \sum_{p=1}^{n_x} \left( z_p - z_{p+1}\right)^2\\
=& \ - \sum_{p=1}^{n_x} \left( \sum_{q=1}^{n_x} D_{pq}^+ z_q \right)^2.\\
\end{align*}
\end{proof}

With these properties at hand, we can now show dissipation of the total energy:
\begin{theorem}\label{th:timecont}
The semi-discrete time-continuous system consisting of (\ref{3_evolution equations spatial discretization}) is energy stable, that is $\dot{E}(t) \leq 0$.
\end{theorem}
\begin{proof}
Let us start from the $\textit{S}$-step in \eqref{3_spatial discretization Sstep}
\begin{align*}
\dot{\widehat S}_{k\ell}(t) = &- \widehat X_{pk}^1 D^x_{pq} \widehat X_{qi}^1 \widehat S_{ij}(t) \widehat V^1_{nj} A_{mn} \widehat V^1_{m\ell} + \widehat X_{pk}^1 D^{xx}_{pq} \widehat X_{qi}^1 \widehat S_{ij}(t) \widehat V^1_{nj} |A|_{mn} \widehat V^1_{m\ell}\\
 &+ \sigma \left( \widehat X_{pk}^1 (x) B_p(t) \widehat V_{0 \ell}^1  - \widehat S_{k\ell}(t)\right).
\end{align*}
We multiply with $\widehat X_{\alpha k}^1 \widehat V_{\beta \ell}^1$, where $\alpha=1,...,n_x$ and $\beta = 0,...,N-1$, sum over $k$ and $\ell$ and introduce the projections $P^{X,1}_{\alpha p} = \widehat X_{\alpha k}^1 \widehat X_{pk}^1$ and $P^{V,1}_{m \beta} = \widehat V_{m \ell}^1 \widehat V_{\beta \ell}^1$. With the notation $\widehat X_{qi}^1 \widehat S_{ij}(t) \widehat V^1_{nj} = u_{qn}^1(t)$ we get
\begin{align*}
\dot u_{\alpha \beta}^1(t)= &- P^{X,1}_{\alpha p} D^x_{pq} u_{qn}^1(t) A_{mn} P^{V,1}_{m \beta} +P^{X,1}_{\alpha p} D^{xx}_{pq} u_{qn}^1(t) |A|_{mn} P^{V,1}_{m \beta}\\
&+ \sigma \left( P^{X,1}_{\alpha p} B_p(t) \delta_{0m} P^{V,1}_{m \beta} - u_{\alpha \beta}^1(t)\right).
\end{align*}
Next, we multiply with $u_{\alpha \beta}^1(t)$ and sum over $\alpha$ and $\beta$. Note that 
\begin{align*}
P^{X,1}_{\alpha p} u_{\alpha \beta}^1 (t) = u_{p \beta}^1 (t) \quad \text{ and } \quad P^{V,1}_{m \beta} u_{p \beta}^1 (t) = u_{pm}^1 (t).
\end{align*}
This leads to
\begin{align*}
\frac{1}{2}\frac{\mathrm{d}}{\mathrm{d}t} \Vert \mathbf{u}^1(t) \Vert^2_F = &- u_{pm}^1(t) D^x_{pq} u_{qn}^1(t) A_{mn} + u_{pm}^1(t) D^{xx}_{pq} u_{qn}^1(t) |A|_{mn}\\
&+  \sigma \left( u_{pm}^1(t) B_p(t) \delta_{0m} - \Vert \mathbf{u}^1(t) \Vert^2 \right).
\end{align*}
Recall that we can write $\mathbf{A} = \mathbf{QMQ}^\top$ with $\mathbf{M} = \diag(\sigma_1,...,\sigma_N)$. Inserting this representation gives
\begin{align*}
\frac{1}{2}\frac{\mathrm{d}}{\mathrm{d}t} \Vert \mathbf{u}^1(t) \Vert^2_F =& - u_{pm}^1(t) D^x_{pq} u_{qn}^1(t) Q_{nk}\sigma_k Q_{mk} + u_{pm}^1(t) D^{xx}_{pq} u_{qn}^1(t) Q_{nk}|\sigma_k| Q_{mk}\\
&+  \left( u_{pm}^1(t) B_p(t) \delta_{0m} - \Vert \mathbf{u}^1(t) \Vert^2 \right)\\
=& - \sigma_k \widetilde u_{pk}^1(t) D^x_{pq} \widetilde u_{qk}^1(t) + |\sigma_k| \widetilde u_{pk}^1(t) D^{xx}_{pq} \widetilde u_{qk}^1(t)\\ &+  \left( u_{pm}^1(t) B_p(t) \delta_{0m} - \Vert \mathbf{u}^1(t) \Vert^2 \right),
\end{align*}
where $\widetilde u_{pk}^1(t) = u_{pm}^1(t)Q_{mk}$.
With the properties of the stencil matrices we get
\begin{align}\label{3_Energystability-u}
\frac{1}{2}\frac{\mathrm{d}}{\mathrm{d}t} \Vert \mathbf{u}^1(t) \Vert^2_F =&  -\left(D^+_{pq} u_{qm}^1(t) |A|_{mn}^{1/2} \right)^2 + \sigma \left(u_{p0}(t) B_p(t) - \Vert \mathbf{u}^1(t) \Vert^2_F\right).
\end{align}
Next we consider equation \eqref{3_spatial discretization B}. Multiplication with $B_p(t)$ and summation over $p$ gives
\begin{align}\label{3_Energystability-B}
\frac{1}{2}\frac{\mathrm{d}}{\mathrm{d}t}  \Vert \mathbf{B}(t)\Vert^2_E &= \sigma \left(u_{p0}(t)B_p(t)-\Vert \mathbf{B}(t)\Vert^2_E\right).
\end{align}
For the total energy of the system it holds that $E(t)= \frac{1}{2} \Vert \mathbf{u}^1(t) \Vert_F^2 + \frac{1}{2} \Vert \mathbf{B}(t)\Vert^2_E $. Adding the evolution equations \eqref{3_Energystability-u} and \eqref{3_Energystability-B} we get
\begin{align*}
\frac{\mathrm{d}}{\mathrm{d}t} E(t) =& - \left( \widetilde D^+_{pq} u_{qm}^1(t) |A|_{mn}^{1/2} \right)^2 + \sigma \left(u_{p0}^1(t) B_p(t) - \Vert \mathbf{u}^1(t) \Vert_F ^2\right)\\
&+ \sigma \left(u_{p0}^1(t)B_p(t)-\Vert \mathbf{B}(t)\Vert_E^2 \right)\\
=& - \left(\widetilde D^+_{pq} u_{qm}^1(t) |A|_{mn}^{1/2} \right)^2 - \sigma \left( (u_{p0}^1(t)-B_p(t))^2 + (u_{pm}^1(t))^2 (1-\delta_{m0})\right),
\end{align*}
where we rewrote $\Vert \mathbf{B}(t) \Vert_E^2 = B_{p}(t)^2$ and $\Vert \mathbf{u}^1(t) \Vert^2_F = (u_{pm}^1(t))^2$. This expression is strictly negative which means that $E$ is dissipated in time. Hence, the system is energy stable.
\end{proof}

\section{Time discretization}\label{sec:time}
Our goal is to construct a conservative DLRA scheme which is energy stable under a sharp time step restriction. Constructing time discretization schemes which preserve the energy dissipation shown in Theorem~\ref{th:timecont} while not suffering from the potentially stiff opacity term is not trivial. In fact a naive IMEX time discretization potentially will increase the total energy, which we demonstrate in the following.

\subsection{Naive time discretization}
We start from system \eqref{3_evolution equations spatial discretization} which still depends continuously on the time $t$. For the time discretization we choose a naive IMEX Euler scheme where we perform a splitting of nternal energy and radiation transport equation. That is, we use an explicit Euler step for the transport part of the evolution equations, treat the internal energy $B$ explicitly and use an implicit Euler step for the radiation absorption term. Note that the scheme describes the evolution from time $t_0$ to time $t_1 = t_0 +\Delta t$ but holds for all further time steps equivalently. This yields the fully discrete scheme
\begin{subequations}\label{eq:naivescheme}
\begin{align}
K_{pk}^1 =&\, K_{pk}^0 -\Delta t D^x_{qp} K_{pj}^0 V_{nj}^0 A_{mn} V_{mk}^0 + \Delta t D^{xx}_{qp} K_{pj}^0 V_{nj}^0 |A|_{mn}  V_{mk}^0 \\\nonumber
&+ \sigma \left( \Delta t B_p^0 V_{0k}^0 - \Delta t K_{pk}^1 \right), \\
L_{mk}^1 =&\, L_{mk}^0 -  \Delta t X_{qk}^0 D_{qp}^x X_{pi}^0 L_{ni}^0 A_{mn} + \Delta t X_{qk}^0 D_{qp}^{xx} X_{pi}^0 L_{ni}^0 |A|_{mn}\\\nonumber
&+ \sigma \left(\Delta t X_{pk}^0 B_p^0 \delta_{m0} - \Delta t L_{mk}^1 \right).
\end{align}
We perform a QR-decomposition of the quantities $[K_{pk}^1, X_{pk}^0]$ and $[ L_{pk}^1, V_{pk}^0]$ to obtain the augmented and time updated bases $\widehat X_{pk}^1$ and $\widehat V_{pk}^1$ according to the rank-adaptive BUG integrator \cite{ceruti2022rank}. Lastly, we perform a Galerkin step for the augmented bases according to
\begin{align}\label{eq:Snaive}
\widehat S_{k\ell}^1 = \widetilde S_{k\ell}^0 &- \Delta t \widehat X_{pk}^1 D^x_{pq} \widehat X_{qi}^1 \widetilde S_{ij}^0  \widehat V_{nj}^1 A_{mn}  \widehat V_{m\ell}^1 + \Delta t \widehat X_{pk}^1 D^{xx}_{pq} \widehat X_{qi}^1 \widetilde S_{ij}^0 \widehat V_{nj}^1 |A|_{mn} \widehat V_{m\ell}^1 \\ \nonumber
&+ \sigma \left(\Delta t \widehat X_{pk}^1 B_p^0 \widehat V_{0 \ell}^1  - \Delta t \widehat S_{k\ell}^1 \right),
\end{align}
where $\widetilde S_{k\ell}^0 := \widehat X_{pk}^1 X_{pi}^0 S_{ij}^0 V_{nj}^0 \widehat V_{n\ell}^1$. The internal energy is then updated via
\begin{align}
B_p^1 &= B_p^0 + \sigma \Delta t \left(\widehat X_{pi}^1 \widehat S_{ij}^1 \widehat V_{0j}^1-B_p^1\right).\label{eq:naiveschemeB}
\end{align}
\end{subequations}

However, this numerical method has the undesirable property that it can increase the total energy during a time step. In Theorem \ref{theorem:naive} we show this analytically.  This behavior is, obviously, completely unphysical. 

\begin{theorem}\label{theorem:naive}
Let $\mathbf{u}^0 \in \mathbb{R}^{n_x \times N}$ with entries $u^0_{pm} = X^0_{pk} S^0_{k \ell} V^0_{m\ell}$ denote the angularly and spatially discretized low-rank approximation of the function $f$ at time $t=t_0$, and $\mathbf{u}^1 \in \mathbb{R}^{n_x \times N}$ with entries $u^1_{\alpha \beta} = \widehat X^1_{\alpha k} \widehat S^1_{k \ell} \widehat V^1_{\beta \ell}$ denote the basis augmented angularly and spatially discretized low-rank approximation at time $t=t_1$ using the rank-adaptive BUG integrator. Further, $\mathbf{B}^0 \in \mathbb{R}^{n_x}$ shall denote the spatially discretized low-rank approximation of $B$ at time $t=t_0$, and $\mathbf{B}^1 \in \mathbb{R}^{n_x}$ at time $t=t_1$, respectively. The total energy at time $t=t_0$ is denoted by $E^0$ and $E^1$ at time $t=t_1$, respectively. Then, there exist initial value pairs $({\mathbf u}^0, \mathbf{B}^0)$ and time step sizes $\Delta t$ such that the naive scheme \eqref{eq:naivescheme} results in $(\mathbf{u}^1, \mathbf{B}^1)$ for which the total energy increases, i.e. for which $E^1 > E^0$.
\end{theorem}

\begin{proof}
Let us multiply the $\textit{S}$-step \eqref{eq:Snaive} with $\widehat X_{\alpha k}^1 \widehat V_{\beta \ell}^1$ and sum over $k$ and $\ell$. Again we make use of the projections $P^{X,1}_{\alpha p} = \widehat X_{\alpha k}^1 \widehat X_{pk}^1$ and $P^{V,1}_{m \beta} = \widehat V_{m \ell}^1 \widehat V_{\beta \ell}^1$. With the definition of $\widetilde S_{k\ell}^{0}$ we obtain
\begin{align}\label{eq:fullschemeUnstable}
u_{\alpha \beta}^1 =& u_{pm}^0 - P^{X,1}_{\alpha p} \Delta t D^x_{pq} u_{qn}^0 A_{mn} P^{V,1}_{m \beta} + P^{X,1}_{\alpha p} \Delta t D^{xx}_{pq} u_{qn}^0 |A|_{mn} P^{V,1}_{m \beta}\\\nonumber
&+ \sigma \left( \Delta t P^{X,1}_{\alpha p} B_p^0 \delta_{m0} P^{V,1}_{m \beta} - \Delta t u_{\alpha \beta}^1 \right).
\end{align}

Let us choose a constant solution in space, i.e., $B^1_p = B^1$ and $u^1_{\alpha \beta} = u^1\delta_{\beta 0}$ for all spatial indices $p, \alpha = 1,...,n_x$. The scalar values $B^1$ and $u^1$ are chosen such that $B^1 = u^1 + \alpha$ where 
\begin{align*}
    0<\alpha < \frac{\sigma \Delta t}{1+ \sigma \Delta t + \sigma^2 \Delta t^2  + \frac12\sigma^3 \Delta t^3} u^1.
\end{align*}
We can now verify that we obtain our chosen values for $B^1_p$ and $u^1_{\alpha \beta}$ after a single step of \eqref{eq:fullschemeUnstable} when using the initial condition
\begin{subequations}\label{eq:initUnstable}
    \begin{align}
   B^0_p =& B^{1} + \sigma\Delta t \alpha = u^1+\alpha (1 + \sigma\Delta t), \\
    u_{pm}^0 =& \left( u^1+\sigma \Delta t (u^1 - B^0_p)\right) \delta_{m0} = \left(u^1-\sigma \Delta t \alpha (1 + \sigma\Delta t) \right) \delta_{m0}.
\end{align}
\end{subequations}
To show this, note that since the solution is constant in space, all terms containing the stencil matrices $\mathbf{D}^x$ and $\mathbf{D}^{xx}$ drop out and we are left with
\begin{align}\label{eq:fullschemeUnstable2}
u_{\alpha \beta}^1 &= u_{pm}^0 + \sigma \left( \Delta t P^{X,1}_{\alpha p} B_p^0 \delta_{m0} P^{V,1}_{m \beta}-  \Delta t u_{\alpha \beta}^1 \right).
\end{align}
Since $B_p^0$ is constant in space and $\delta_{m0}$ lies in the span of our basis, we know that all projections in the above equation are exact. Plugging the initial values \eqref{eq:initUnstable} into \eqref{eq:fullschemeUnstable2} we then directly obtain $u_{\alpha \beta}^1 = u^1 \delta_{\beta 0}$. Similarly, by plugging \eqref{eq:initUnstable} into \eqref{eq:naiveschemeB}, we obtain $B^1_p = B^1$.

Then, we square both of the initial terms \eqref{eq:initUnstable} to get
\begin{align*}
    (B^0_p)^2 =& (B^1)^2 + 2\sigma\Delta t \alpha B^1 + \sigma^2\Delta t^2 \alpha^2 = (B^1)^2 + 2\sigma\Delta t \alpha (u^1 + \alpha) + \sigma^2\Delta t^2 \alpha^2, \\
    (u_{pm}^0)^2 =& \left( (u^1)^2-2\sigma \Delta t \alpha u^1 (1 + \sigma\Delta t) + \sigma^2 \Delta t^2 \alpha^2 (1 + \sigma\Delta t)^2 \right) \delta_{m0}.
\end{align*}
Summing over $p$ and $m$, adding these two terms and multiplying with $\frac12$ yields
\begin{align*}
    E^1 = E^0 + \sigma^2 \Delta t^2\alpha u^1 - \sigma \Delta t \alpha^2- \frac12\sigma^2 \Delta t^2\alpha^2 - \frac12\sigma^2 \Delta t^2\alpha^2 (1+\sigma \Delta t)^2.
\end{align*}
Note that $E^1 > E^0$ if
\begin{align*}
    \sigma \Delta t u^1 - \alpha- \frac12\sigma \Delta t\alpha - \frac12\sigma \Delta t\alpha (1+\sigma \Delta t)^2 > 0.
\end{align*}
Rearranging gives
\begin{align*}
    \alpha < \frac{\sigma \Delta t}{1+ \sigma \Delta t + \sigma^2 \Delta t^2  + \frac12\sigma^3 \Delta t^3} u^1.
\end{align*}
This is exactly the domain $\alpha$ is chosen from. Hence, we have $E^1 > E^0$, which is the desired result.
\end{proof}

\subsection{Energy stable space-time discretization}
We have seen that the naive scheme presented in \eqref{eq:naivescheme} can increase the total energy in one time step. The main goal of this section is to construct a novel energy stable time integration scheme for which the corresponding analysis leads to a classic hyperbolic CFL condition that enables us to operate up to a time step size of $\Delta t = \text{CFL} \cdot \Delta x$. For constructing this energy stable scheme, we write the original equations in two parts followed by a basis augmentation and correction step. 

In detail, we first solve
\begin{subequations}\label{eq:schemeStable}
\begin{align}
K_{pk}^{\star} =&\, K_{pk}^{0} -\Delta t D^x_{qp} K_{pj}^{0} V_{nj}^{0} A_{mn} V_{mk}^{0} + \Delta t D^{xx}_{qp} K_{pj}^{0} V_{nj}^{0} |A|_{mn} V_{mk}^{0}, \label{eq:schemeStableK} \\
L_{mk}^{\star} =&\,  L_{mk}^{0} -  \Delta t X_{qk}^{0} D_{qp}^x X_{pi}^{0} L_{ni}^{0} A_{mn} + \Delta t X_{qk}^{0} D_{qp}^{xx} X_{pi}^{0} L_{ni}^{0} |A|_{mn}.\label{eq:schemeStableL}
\end{align}
We perform a QR-decomposition of the augmented quantities $\mathbf{X}^{\star} \mathbf{R} = [\mathbf{K}^{\star}, \mathbf{X}^{0}]$ and $\mathbf{V}^{\star} \mathbf{\widetilde R} = [\mathbf{L}^{\star}, \mathbf{V}^{0}]$ to obtain the augmented and time updated bases $\mathbf{X}^\star$ and $\mathbf{V}^\star$. Note that $\mathbf{R}$ and $\mathbf{\widetilde R}$ are discarded. With $\widetilde S_{\alpha \beta}^{0} = X^{\star}_{j\alpha}X_{j\ell}^{0}S_{\ell m}^{0}V_{km}^{0}V_{k\beta}^{\star}$ we then solve the $\textit{S}$-step equation
\begin{align}\label{eq:schemeStableS}
S_{\alpha \beta}^{\star} = \widetilde S_{\alpha \beta}^{0} - &\Delta t X_{p\alpha}^{\star} D^x_{pq} X_{qi}^{\star} \widetilde S_{ij}^{0} V_{nj}^{\star} A_{mn} V_{m \beta}^{\star} + \Delta t X_{p\alpha}^{\star} D^{xx}_{pq} X_{qi}^{\star} \widetilde S_{ij}^{0} V_{nj}^{\star} |A|_{mn} V_{m \beta}^{\star}.
\end{align}
Second, we solve the coupled equations for the internal energy $\mathbf{B} \in \mathbb{R}^{n_x}$ and the quantity $\mathbf{\widehat u_{0}}^{1} = (\widehat u_{j0}^{1})_j \in \mathbb{R}^{n_x}$ to which we refer as the zeroth order moment according to
\begin{align}\label{eq:schemeStableu0}
    \widehat u_{j0}^{1} =&\, X_{j\ell}^{0} S_{\ell m}^{0} V_{0m}^{0} - \Delta t D^x_{ji} X_{in}^{\star} \widetilde S_{nm}^{0} V_{\ell m}^{\star} A_{0\ell}+\Delta t D^{xx}_{ji}X_{in}^{\star} \widetilde S_{nm}^{0} V_{\ell m}^{\star} |A|_{0\ell}\\\nonumber
 &+ \sigma\Delta t(B_j^{1}-\widehat u_{j0}^{1}),\\
    B_j^{1} =&\, B_j^0 + \sigma\Delta t(\widehat u_{j0}^{1}-B_j^{1})\label{eq:schemeStableB}.
\end{align}
Following \cite[Section~6]{kusch2023stability} we perform the opacity update only on $\mathbf{L}= \mathbf{V}^{\star} \mathbf{S}^{\star}$ according to 
\begin{align}
    L_{mk}^{\star, \text{scat}} = \frac{1}{1+\Delta t \sigma} L_{mk} \quad \text{ for } k \ne 0\label{eq:schemeStableScattering}
\end{align}
and perform a QR-decomposition $ \mathbf{V}^{\star, \text{scat}} \mathbf{S}^{\star, \text{scat}, \top} = \mathbf{L}^{\star, \text{scat}}$ to retrieve the factorized basis $\mathbf{V}^{\star, \text{scat}}$ and the coefficients from the matrix $\mathbf{S}^{\star, \text{scat}}$. We then augment the basis matrices according to 
\begin{align}\label{eq:schemeStableAugmentu0}
    \mathbf{\widetilde X}^{1} = \text{qr}([\mathbf{\widehat u}_0^1, \mathbf{X}^{\star}]), \quad \mathbf{\widetilde V}^{1} = \text{qr}([\mathbf{e}_1, \mathbf{V}^{\star, \text{scat}}]).
\end{align}
Third, the coefficient matrix is updated via
\begin{align}\label{eq:schemeStableScorrect}
    \mathbf{\widetilde S}^{1} = \mathbf{\widetilde X}^{1,\top}\mathbf{X}^{\star}\mathbf{S}^{\star, \text{scat}}\mathbf{V}^{\star, \text{scat}, \top} (\mathbf{I} - \mathbf{e}_1\mathbf{e}_1^{\top})\mathbf{\widetilde V}^{1} + \mathbf{\widetilde  X}^{1,\top}\mathbf{\widehat u}_0^1\mathbf{e}_{1,\top}\mathbf{\widetilde V}^{1} \, \in\mathbb{R}^{(2r + 1)\times (2r + 1)}.
\end{align}
Then, we obtain the updated solution $\mathbf{\widetilde X}^1 \mathbf{\widetilde S}^1 \mathbf{\widetilde V}^{1,\top} \in \mathbb{R}^{n_x \times N}$. Lastly, we truncate this rank $2r + 1$ solution to a new rank $r_1$ using a suited truncation strategy such as proposed in \cite{ceruti2022rank} or the conservative truncation strategy of \cite{einkemmer2022robust}. This finally gives the low-rank factors $\mathbf{X}^1, \mathbf{S}^1$ and $\mathbf{V}^1$.
\end{subequations}
We show that the given scheme is energy stable and start with the following Lemma.

\begin{lemma}\label{Lemma:stencil_matrices_Fourier}
Let us denote $u_{jk}^1 := \widetilde X_{j\alpha}^1 \widetilde S_{\alpha \beta}^1 \widetilde V_{k\beta}^1$. Under the time step restriction $\Delta t \leq \Delta x$ it holds
\begin{align}\label{eq:Lemma2}
\frac{\Delta t}{2}(D^x_{ji}u_{jk}^{1}A_{k\ell} - D^{xx}_{ji}u_{jk}^{1}|A|_{k\ell})^2-\left(D_{ji}^+ u_{ik}^1|A|_{k\ell}^{1/2}\right)^2 \leq 0.
\end{align}
\end{lemma}
\begin{proof}
Following \cite{kusch2023stability}, we employ a Fourier analysis which allows us to write the stencil matrices $\mathbf{D}^{x,xx,+}$ in diagonal form. Let us define 
$\mathbf E
\in\mathbb{C}^{n_x\times n_x}$ with entries
\begin{align*}
E_{k\alpha} = \sqrt{\Delta x}\exp(i\alpha\pi x_{k}), \quad k,\alpha = 1,...,n_x
\end{align*}
with $i\in\mathbb{C}$ being the imaginary unit. Then, the matrix $\mathbf{E}$ is orthonormal, i.e., $\mathbf E\mathbf E^H = \mathbf E^H\mathbf E = \mathbf I$ (the uppercase $H$ denotes the complex transpose) and it diagonalizes the stencil matrices:
\begin{align}\label{eq:diagScheme}
\mathbf{D}^{x,xx,+}\mathbf E = \mathbf E\mathbf{\Lambda}^{x,xx,+}\;.
\end{align}
The matrices $\mathbf{\Lambda}^{x,xx,+}$ are diagonal with entries
\begin{align*}
    \lambda^{x}_{\alpha,\alpha} =& \frac{1}{2\Delta x}(e^{i\alpha\pi \Delta x}-e^{-i\alpha\pi \Delta x}) = \frac{i}{\Delta x} \sin(\omega_{\alpha})\;, \\
    \lambda_{\alpha,\alpha}^{xx} =& \frac{1}{2\Delta x} \left(e^{i\alpha\pi \Delta x}-2+e^{-i\alpha\pi \Delta x}\right) = \frac{1}{\Delta x} \left(\cos(\omega_{\alpha})-1\right)\;,\\
    \lambda_{\alpha,\alpha}^{+} =& \frac{1}{\sqrt{2\Delta x}} \left(e^{i\alpha\pi \Delta x}-1\right) = \frac{1}{\sqrt{2\Delta x}} \left(\cos(\omega_{\alpha}) + i\sin(\omega_{\alpha})-1\right)\;,
\end{align*}
where we use $\omega_{\alpha}:=\alpha\pi \Delta x$. Moreover, recall that we can write $\mathbf A = \mathbf Q\mathbf M \mathbf Q^{\top}$ where $\mathbf{M} = \text{diag}(\sigma_1,\cdots,\sigma_N)$. We then have with $\widehat u_{jk} = E_{j\ell}u_{\ell m}Q_{mk}$
\begin{align*}
    &\frac{\Delta t}{2} (D^x_{ji}u_{jk}^{1}A_{k\ell} - D^{xx}_{ji}u_{jk}^{1}|A|_{k\ell})^2-\left(D_{ji}^+ u_{ik}^{1}|A|_{k\ell}^{1/2}\right)^2
    \\
    =&\frac{\Delta t}{2} \left|\lambda^x_{jj}\widehat u_{jk}^{1}\sigma_{k} - \lambda^{xx}_{jj}\widehat u_{jk}^{1}|\sigma_{k}|\right|^2-\left|\lambda^+_{jj} \widehat u_{jk}^{1}|\sigma_{k}|^{1/2}\right|^2 \\
    \leq & \left[\Delta t\left( \frac{|\sigma_k|^2}{\Delta x^2}\cdot\left|1-\cos(\omega_j)\right|\right) - \frac{|\sigma_k|}{\Delta x} \cdot\left|1-\cos(\omega_j)\right|\right]  (\widehat u_{jk}^1)^2.
\end{align*}
To ensure negativity, we must have
\begin{align*}
    \Delta t \left( \frac{|\sigma_k|^2}{\Delta x^2} \cdot\left|1-\cos(\omega_j)\right|\right) \leq \frac{|\sigma_k|}{\Delta x} \cdot\left|1-\cos(\omega_j)\right|.
\end{align*}
Hence, for $\Delta t \leq \frac{\Delta x}{|\sigma_k|}$ equation \eqref{eq:Lemma2} holds. Since $|\sigma_k| \leq 1$, we have proven the Lemma.
\end{proof}
We can now show energy stability of the proposed scheme:
\begin{theorem}
Under the time step restriction $\Delta t \leq \Delta x$, the scheme \eqref{eq:schemeStable} is energy stable, i.e.,
\begin{align}
    \Vert \mathbf B^{1}\Vert_E^2 + \Vert \mathbf{ X}^{1}\mathbf{ S}^1\mathbf{ V}^{1,\top}\Vert_F^2 \leq \Vert \mathbf B^{0}\Vert_E^2 + \Vert \mathbf X^{0}\mathbf{S}^0\mathbf{V}^{0,\top}\Vert_F^2.
\end{align}
\end{theorem}

\begin{proof}
First, we multiply \eqref{eq:schemeStableB} with $B_j^{1}$ and sum over $j$. Then,
\begin{align*}
    \left(B_j^{1}\right)^2 =& B_j^0B_j^{1} + \sigma\Delta t \left(u_{j0}^{1} B_j^{1}-\left(B_j^{1}\right)^2\right).
\end{align*}
Let us note that
\begin{align*}
   B_j^{0}B_j^{1} = \frac{\left(B_j^{1}\right)^2}{2} + \frac{\left(B_j^{0}\right)^2}{2} - \frac12(B_j^{1}-B_j^{0})^2.
\end{align*}
Hence,
\begin{align}\label{eq:discreteB2}
    \frac12\left(B_j^{1}\right)^2 =& \frac12\left(B_j^{0}\right)^2 - \frac12(B_j^{1}-B_j^{0})^2 + \sigma\Delta t \left(u_{j0}^{1} B_j^{1}-\left(B_j^{1}\right)^2\right).
\end{align}

 To obtain a similar expression for $(u_{jk}^{1})^2$, we multiply \eqref{eq:schemeStableS} with $X_{j \alpha}^{\star} V_{k   \beta}^{\star}$ and sum over $\alpha$ and $\beta$. For simplicity of notation, let us define $u_{jk  }^{\star} := X_{j  \alpha }^{\star}S_{\alpha \beta }^{\star} V_{k  \beta}^{\star}$ and $u_{jk  }^{0} := X_{j \alpha}^{\star}\widetilde S_{\alpha \beta}^{0} V_{k  \beta}^{\star}$ as well as the projections $P^X_{j p} := X_{j \alpha}^{\star}X_{p\alpha}^{\star}$ and $P^V_{k   m} := V_{k  \beta}^{\star} V_{m\beta}^{\star}$. Then, we obtain the system
\begin{align}\label{eq:schemeu}
    u_{j  k  }^{\star} = u_{j  k  }^{0} - \Delta t P^X_{j  p} D^x_{pq} u_{qn}^{0} A_{mn} P^V_{k   m} + \Delta t P^X_{j  p} D^{xx}_{pq} u_{qn}^{0} |A|_{mn} P^V_{k   m}.
\end{align}
Next, we define $u_{jk}^1 := \widetilde X_{j\alpha}^1 \widetilde S_{\alpha \beta}^1 \widetilde V_{k\beta}^1$ and note that by construction we have that 
\begin{align*}
    u_{jk}^1 = \frac{u_{j  k  }^{\star} (1-\delta_{k0})}{1+\sigma \Delta t} + \widehat u_{j  0  }^{1}\delta_{k0}.
\end{align*}
Hence, plugging in the schemes for $u_{j  k  }^{\star}$ and $ \widehat u_{j0}^1$, that is, \eqref{eq:schemeu} and \eqref{eq:schemeStableu0} we get
\begin{align*}
    (1+\sigma\Delta t)u_{jk}^1 =&\, \left( u_{j  k  }^{0} - \Delta t P^X_{j  p} D^x_{pq} u_{qn}^{0} A_{mn} P^V_{k   m} + \Delta t P^X_{j  p} D^{xx}_{pq} u_{qn}^{0} |A|_{mn} P^V_{km}\right) (1-\delta_{k0})\\
    &+ \Bigl( X_{j\ell}^{0}S_{\ell m}^{0} V_{0m}^{0} - \Delta tD^x_{ji}X_{in}^{\star} \widetilde S_{nm}^{0} V_{\ell m}^{\star} A_{0\ell}+\Delta tD^{xx}_{ji}X_{in}^{\star} \widetilde S_{nm}^{0} V_{\ell m}^{\star} |A|_{0\ell}\\
    &+ \sigma\Delta t B_j^{1} \Bigr)  \delta_{k0}.
\end{align*}
Let us note that $P^V_{km}P^X_{jp}u_{jk}^{1} = u_{jk}^{1}$ for $k\neq 0$. Hence, multiplying the above equation with $u_{jk}^{1}$ and summing over $j$ and $k$ gives
\begin{align*}
    \frac12\left(u_{jk}^{1}\right)^2 = \frac12\left(u_{jk}^{0}\right)^2 - \frac12(u_{jk}^{1}-u_{jk}^{0})^2-&\Delta t u_{jk}^{1}D^x_{ji}u_{i\ell}^0 A_{k\ell}+\Delta tu_{jk}^{1}D^{xx}_{ji}u_{i\ell}^0 |A|_{k\ell}\\
    +& \sigma\Delta tu_{jk}^{1}(B_j^{1}\delta_{k0}-u_{jk}^{1}).
\end{align*}

Let us now add the zero term $\Delta tu_{jk}^{1}D^x_{ji}u_{i\ell}^{1} A_{k\ell}$ and add and subtract the term
$\Delta tu_{jk}^{1}D^{xx}_{ji}u_{i\ell}^{1} |A|_{k\ell}$. Then,
\begin{align*}
    \frac12\left(u_{jk}^{1}\right)^2 = \frac12\left(u_{jk}^{0}\right)^2 - \frac12(u_{jk}^{1}-u_{jk}^{0})^2-&\Delta tu_{jk}^{1}D^x_{ji}(u_{i\ell}^0-u_{i\ell}^{1}) A_{k\ell}\\
    +&\Delta tu_{jk}^{1}D^{xx}_{ji}(u_{i\ell}^0-u_{i\ell}^{1}) |A|_{k\ell}+\Delta tu_{jk}^{1}D^{xx}_{ji}u_{i\ell}^{1} |A|_{k\ell}\\
    +& \sigma\Delta tu_{jk}^{1}(B_j^{1}\delta_{k0}-u_{jk}^{1}).
\end{align*}
In the following, we use Young's inequality which states that for $a,b\in\mathbb{R}$ we have $a\cdot b \leq \frac{a^2}{2} + \frac{b^2}{2}$. We now apply this to the term
\begin{align*}
    -\Delta tu_{jk}^{1}D^x_{ji}(u_{i\ell}^0-u_{i\ell}^{1}) A_{k\ell}+\Delta tu_{jk}^{1}D^{xx}_{ji}(u_{i\ell}^0-u_{i\ell}^{1}) |A|_{k\ell}\\
    \leq\frac12 (u_{i\ell}^0-u_{i\ell}^{1})^2 + \frac{\Delta t^2}{2}(D^x_{ji}u_{jk}^{1}A_{k\ell} - D^{xx}_{ji}u_{jk}^{1}|A|_{k\ell})^2.
\end{align*}
Hence, using $u_{jk}^{1}D^{xx}_{ji}u_{i\ell}^{1} |A|_{k\ell}=-\left(D_{ji}^+ u_{ik}^{1}|A|_{k\ell}^{1/2}\right)^2$ we get
\begin{align}\label{eq:discreteu2}
 \frac12\left(u_{jk}^{1}\right)^2 \leq \frac12\left(u_{jk}^{0}\right)^2 +& \frac{\Delta t^2}{2}(D^x_{ji}u_{jk}^{1}A_{k\ell} - D^{xx}_{ji}u_{jk}^{1}|A|_{k\ell})^2-\Delta t\left(D_{ji}^+ u_{ik}|A|_{k\ell}^{1/2}\right)^2\nonumber\\
 +& \sigma\Delta tu_{jk}^{1}(B_j^{1}\delta_{k0}-u_{jk}^{1}).
\end{align}
As for the continuous case, we add \eqref{eq:discreteu2} and \eqref{eq:discreteB2} to obtain a time update equation for $E^0 := \frac12\left(u_{jk}^{0}\right)^2 + \frac12\left(B_{j}^{0}\right)^2$:
\begin{align}\label{eq:proof_discrete_1}
    E^{1} \leq E^{0} +& \frac{\Delta t^2}{2}(D^x_{ji}u_{jk}^{1}A_{k\ell} - D^{xx}_{ji}u_{jk}^{1}|A|_{k\ell})^2-\Delta t\left(D_{ji}^+ u_{ik}^1|A|_{k\ell}^{1/2}\right)^2\nonumber\\
   +& \sigma\Delta t(u_{j0}^{1}B_j^{1}-(u_{jk}^{1})^2) - \frac12(B_j^{1}-B_j^0)^2 
    + \sigma\Delta t\left(u_{j0}^{1} B_j^{1}-\left(B_j^{1}\right)^2\right)\nonumber\\
    \leq E^{0} +& \frac{\Delta t^2}{2}(D^x_{ji}u_{jk}^{1}A_{k\ell} - D^{xx}_{ji}u_{jk}^{1}|A|_{k\ell})^2-\Delta t\left(D_{ji}^+ u_{ik}^1|A|_{k\ell}^{1/2}\right)^2\nonumber\\
    -& \sigma\Delta t(B_j^{1}-u_{jk}^{1})^2 - \frac12(B_j^{1}-B_j^0)^2 .
\end{align}
With Lemma \ref{Lemma:stencil_matrices_Fourier} we have that
\begin{align*}
\frac{\Delta t}{2}(D^x_{ji}u_{jk}^{1}A_{k\ell} - D^{xx}_{ji}u_{jk}^{1}|A|_{k\ell})^2-\left(D_{ji}^+ u_{ik}^1|A|_{k\ell}^{1/2}\right)^2 \leq 0
\end{align*}
for $\Delta t \leq \Delta x$. Since the truncation step is designed to not alter the zero order moments, we conclude that $E^1 \leq E^0$ and the full scheme is energy stable under the time step restriction $\Delta t \leq \Delta x$.
\end{proof}

\section{Mass conservation}\label{sec:mass}
A drawback of dynamical low-rank approximation using the classical integrators introduced in Section \ref{sec:intro} is that the method does not preserve physical invariants. It has been shown in \cite{einkemmerjoseph2021conservative} that this problem can be overcome when using a modified \textit{L}-step equation. On this basis, \cite{einkemmer2022robust, guo2022} have presented conservative DLRA algorithms where they additionally introduced a conservative truncation step. In contrast to \cite{einkemmer2022robust, guo2022} we do not need to consider a modified \textit{L}-step equation due to the applied basis augmentation strategy from \cite{ceruti2022rank}, but use the conservative truncation step. Then we can show that besides being energy stable, our scheme ensures local conservation of mass. The conservative truncation strategy works as follows:
\begin{enumerate}
\item Compute $\mathbf{\widetilde K} = \mathbf{\widetilde X}^1 \mathbf{\widetilde S}^1$ and split it into two parts $\mathbf{\widetilde K} = [\mathbf{\widetilde K}^{\text{cons}}, \mathbf{\widetilde K}^{\text{rem}}]$ where $\mathbf{\widetilde K}^{\text{cons}}$ corresponds to the first and $\mathbf{\widetilde K}^{\text{rem}}$ consists of the remaining columns of $\mathbf{\widetilde K}$. 

Analogously, distribute $\mathbf{\widetilde V}^1 = [\mathbf{\widetilde V}^{\text{cons}}, \mathbf{\widetilde V}^{\text{rem}}]$ where $\mathbf{\widetilde V}^{\text{cons}}$ corresponds to the first and $\mathbf{\widetilde V}^{\text{rem}}$ consists of the remaining columns of $\mathbf{\widetilde V}$. 

\item Derive $\mathbf{X}^{\text{cons}} = \mathbf{\widetilde K}^{\text{cons}} / \Vert \mathbf{\widetilde K}^{\text{cons}}\Vert$ and $\mathbf{S}^{\text{cons}} = \Vert \mathbf{\widetilde K}^{\text{cons}}\Vert$.

\item Perform a QR-decomposition of $\mathbf{\widetilde K}^{\text{rem}}$ to obtain $\mathbf{\widetilde K}^{\text{rem}} = \mathbf{\widetilde X}^{\text{rem}} \mathbf{\widetilde S}^{\text{rem}}$. 

\item Compute the singular value decomposition of $\mathbf{\widetilde S}^{\text{rem}} = \mathbf{U \Sigma W^\top}$ with $\mathbf{\Sigma} = \diag (\sigma_j)$. Given a tolerance $\vartheta$, choose the new rank $r_1 \leq 2r$ as the minimal number such that
\begin{align*}
\left(\sum_{j=r_1+1}^{2r} \sigma_j^2\right)^{1/2} \leq \vartheta.
\end{align*}
Let $\mathbf{S}^{\text{rem}}$ be the $r_1 \times r_1$ diagonal matrix with the $r_1$ largest singular values and let $\mathbf{U}^{\text{rem}}$ and $\mathbf{W}^{\text{rem}}$ contain the first $r_1$ columns of $\mathbf{U}$ and $\mathbf{W}$, respectively. Set $\mathbf{X}^{\text{rem}} = \mathbf{\widetilde X}^{\text{rem}} \mathbf{U}^{\text{rem}}$ and $\mathbf{V}^{\text{rem}} = \mathbf{\widetilde V}^{\text{rem}} \mathbf{W}^{\text{rem}}$.

\item Set $\mathbf{\widehat X} = [\mathbf{X}^{\text{cons}}, \mathbf{X}^{\text{rem}}]$ and $\mathbf{\widehat V} = [\mathbf{e}_1, \mathbf{V}^{\text{rem}}]$. Perform a QR-decomposition of $\mathbf{\widehat X} = \mathbf{X}^1 \mathbf{R}^1$ and $\mathbf{\widehat V} = \mathbf{V}^1 \mathbf{R}^2$.

\item Set
\begin{align*}
\mathbf{S}^1 = \mathbf{R}^1 \begin{bmatrix}
\mathbf{S}^{\text{cons}} & 0 \\ 
0 & \mathbf{S}^{\text{rem}}
\end{bmatrix} \mathbf{R}^{2,\top}.
\end{align*}
The updated solution at time $t_1 = t_0+ \Delta t$ is then given by $\mathbf{u}^1 = \mathbf X^{1}\mathbf{S}^1\mathbf{V}^{1,\top}$.
\end{enumerate}

Then, the scheme is conservative:
\begin{theorem}
The scheme \eqref{eq:schemeStable} is locally conservative. That is, for the scalar flux at time $t_n$ denoted by $\Phi_j^n = X_{j\ell}^nS_{\ell m}^{n} V_{0m}^n$, where $n\in\{0,1\}$ and $u_{jk}^0 = X_{j\ell}^0S_{\ell m}^{0} V_{km}^0$ it fulfills the conservation law
\begin{subequations}\label{eq:localConservation}
\begin{align}
    \Phi^1_j =&\, \Phi^0_j - \Delta tD^x_{ji} u_{i\ell}^0 A_{0\ell}+\Delta tD^{xx}_{ji}u_{i\ell}^0 |A|_{0\ell} + \sigma\Delta t(B_j^{1}-\Phi^1_j),\\
    B_j^{1} =&\, B_j^0 + \sigma\Delta t(\Phi^1_j-B_j^{1}).
\end{align}
\end{subequations}
\end{theorem}
\begin{proof}
The conservatice truncation step is designed such that it does not alter the first column of $ \mathbf{\widetilde X}^1 \mathbf{\widetilde S}^1 \mathbf{\widetilde V}^{1,\top}$. Together with the basis augmentation \eqref{eq:schemeStableAugmentu0} and correction step \eqref{eq:schemeStableScattering} we then know that 
\begin{align*}
    \Phi^1_j = X_{j\ell}^1 S_{\ell m}^1 V_{0m}^1 = \widetilde X_{j\ell}^1 \widetilde S_{\ell m}^1\widetilde V_{0m}^1 = \widehat u_{j0}^1.
\end{align*}
Hence, with \eqref{eq:schemeStableu0} and \eqref{eq:schemeStableB} we get that
\begin{align*}
    \Phi^1_j =&\, X_{j\ell}^{0}S_{\ell m}^{0} V_{0m}^{0} - \Delta tD^x_{ji}X_{in}^{\star} \widetilde S_{nm}^{0} V_{\ell m}^{\star} A_{0\ell}+\Delta tD^{xx}_{ji}X_{in}^{\star} \widetilde S_{nm}^{0} V_{\ell m}^{\star} |A|_{0\ell}\\
    &+ \sigma\Delta t(B_j^{1}-\Phi^1_j),\\
    B_j^{1} =&\, B_j^0 + \sigma\Delta t(\Phi^1_j-B_j^{1}).
\end{align*}
Since the basis augmentation with $\mathbf{X}^0$ and $\mathbf{V}^0$ ensures $X_{j\ell}^{0}S_{\ell m}^{0} V_{0m}^{0} = X_{in}^{\star} \widetilde S_{nm}^{0} V_{\ell m}^{\star} = u_{i\ell}^0$, the local conservation law \eqref{eq:localConservation} holds.
\end{proof}

Hence, equipped with a conservative truncation step, the energy stable algorithm presented in \eqref{eq:schemeStable} conserves mass locally. To give an overview of the algorithm, we visualize the main steps in Figure~\ref{fig:flowchart}. 
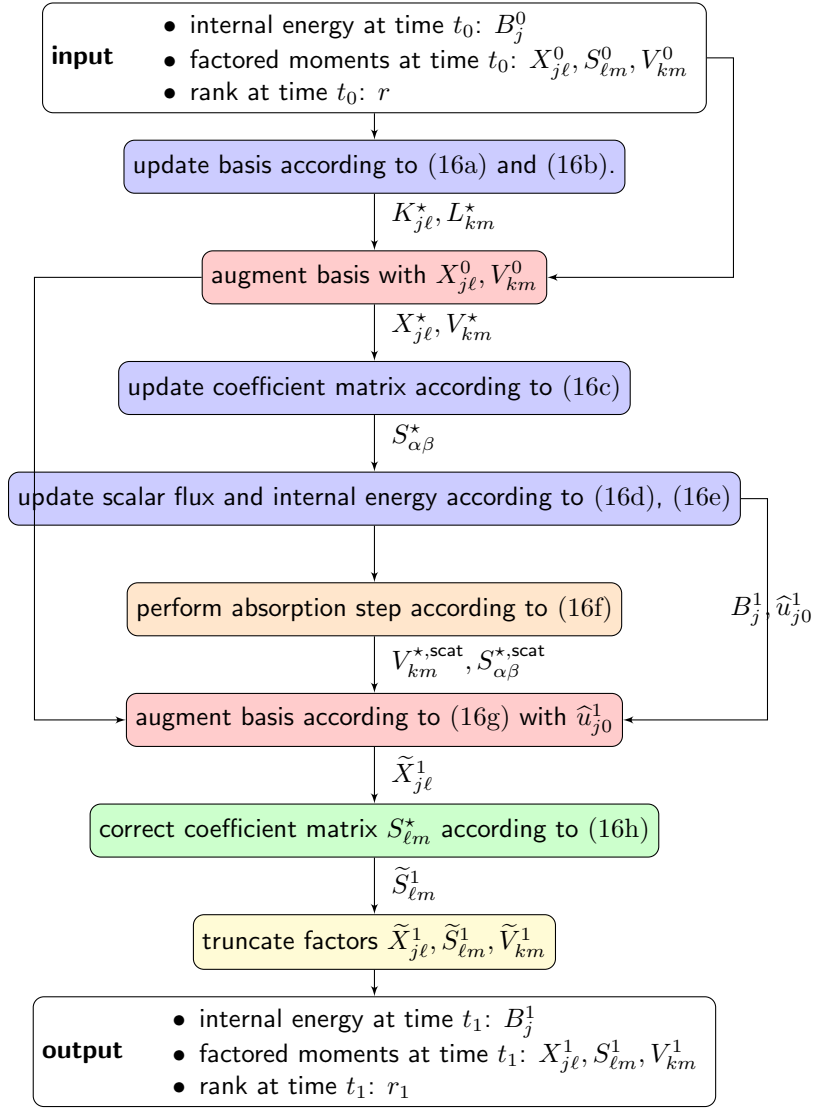
\begin{figure}[htp!]
    \centering
    \begin{tikzpicture}[node distance = 3.5cm,auto,font=\sffamily]
        \node[block](init){
        \textbf{input}
    \begin{varwidth}{\linewidth}\begin{itemize}
        \item internal energy at time $t_0$: $B_j^0$
        \item factored moments at time $t_0$:  $X^0_{j\ell},S^0_{\ell m}, V^0_{km}$
        \item rank at time $t_0$: $r$
    \end{itemize}\end{varwidth}
        };
        \node[block, below = 0.35cm of init, fill=blue!20](KLstep){update basis according to \eqref{eq:schemeStableK} and \eqref{eq:schemeStableL}.};
        \node[block, below = 0.75cm of KLstep,fill=red!20](augment1){augment basis with $X^0_{j\ell}, V^0_{km}$};
        \node[block, below = 0.75cm of augment1,fill=blue!20](SStep){update coefficient matrix according to \eqref{eq:schemeStableS}};
        \node[block, below = 0.75cm of SStep,fill=blue!20](u0Bupdate){update scalar flux and internal energy according to \eqref{eq:schemeStableu0}, \eqref{eq:schemeStableB}};
        \node[block, below = 0.75cm of u0Bupdate,fill=orange!20](absorption){perform absorption step according to \eqref{eq:schemeStableScattering}};
        \node[block, below = 0.75cm of absorption,fill=red!20](augment2){augment basis according to \eqref{eq:schemeStableAugmentu0} with $\widehat u_{j0}^1$};
        \node[block, below = 0.75cm of augment2,fill=green!20](Scorrection){correct coefficient matrix $S_{\ell m}^{\star}$ according to \eqref{eq:schemeStableScorrect}};
        \node[block, below = 0.75cm of Scorrection,fill=yellow!20](truncate){truncate factors $\widetilde X^1_{j\ell},\widetilde S^1_{\ell m},\widetilde V^1_{km}$};
        \node[block, below = 0.35cm of truncate](out){
        \textbf{output}
    \begin{varwidth}{\linewidth}\begin{itemize}
        \item internal energy at time $t_1$: $B_j^1$
        \item factored moments at time $t_1$:  $X^1_{j\ell},S^1_{\ell m},V^1_{km}$
        \item rank at time $t_1$: $r_1$
    \end{itemize}\end{varwidth}
        };
        \path[line] (init) -- node [near end] {} (KLstep);
        \path[line] (KLstep) -- node [near end] {} (augment1);
        \path[line] (augment1) -- node [near end] {} (SStep);
        \path[line] (init.east) -- ([xshift=0.35cm] init.east) |- (augment1.east);
        \path[line] (augment1.west) -- ([xshift=-2.2cm] augment1.west) |- (augment2.west);
        \path[line] (SStep) -- (u0Bupdate);
        \path[line] (u0Bupdate) -- node [near end] {} (absorption);
        \path[line] (u0Bupdate.east) -- ([xshift=0.35cm] u0Bupdate.east) |- (augment2.east);
        \node[right=0.0cm and 1.3cm of absorption.east](test){$B^{1}_{j}, \widehat u^{1}_{j0}$};
        \path[line] (absorption) -- node [near end] {} (augment2);
        \path[line] (augment2) -- (Scorrection);
        \path[line] (Scorrection) -- node [near start] {} (truncate);
        \path[line] (truncate) -- node [near start] {} (out);
        \node[below right=0.0cm and 0.1cm of augment1.south](test){$X^{\star}_{j\ell},V^{\star}_{km}$};
        \node[below right=0.0cm and 0.1cm of SStep.south](test){$S^{\star}_{\alpha \beta}$};
        \node[below right=0.0cm and 0.1cm of KLstep.south](test){$K^{\star}_{j\ell}, L^{\star}_{km}$};
        \node[below right=0.0cm and 0.1cm of absorption.south](test){$V_{km}^{\star, \text{scat}}, S_{\alpha \beta}^{\star, \text{scat}} $};
        \node[below right=0.0cm and 0.1cm of augment2.south](test){$\widetilde X_{j\ell}^1$};
        \node[below right=0.0cm and 0.1cm of Scorrection.south](test){$\widetilde S_{\ell m}^1$};
    \end{tikzpicture}
    \caption{Flowchart of the stable and conservative method \eqref{eq:schemeStable}.}
    \label{fig:flowchart}
\end{figure}

\section{Numerical results}\label{sec:num}

In this section we give numerical results to validate the proposed DLRA algorithm. The source code to reproduce the presented numerical results is openly available, see \cite{baumann2023}.

\subsection{1D Plane source}

We consider the thermal radiative transfer equations as described in \eqref{eq1a} on the spatial domain $D = [-10,10]$. As initial distribution we choose a cutoff Gausian
\begin{align*}
u(t=0,x) = \max \left(10^{-4}, \frac{1}{\sqrt{2\pi\sigma_{\mathrm{IC}}^2}} \exp\left(-\frac{(x-1)^2}{2\sigma_{\mathrm{IC}}^2}\right)\right),
\end{align*}
with constant deviation $\sigma_{\mathrm{IC}}=0.03$. Particles are initially centered around $x=1$ and move into all directions $\mu\in [-1,1]$. The initial value for the internal energy is set to $B^0=1$ and we start computations with a rank of $r=20$. The opacity $\sigma$ is set to the constant value of $1$. Note that this setting is an extension of the so-called \textit{plane source} problem, which is a common test case for the radiative transfer equation \cite{ganapol2008}. In the context of dynamical low-rank approximation it has been studied in \cite{ceruti2022rank, kusch2023stability, peng2021high, peng2020-2D}. We compare the solution of the full coupled-implicit system without DLRA which reads
\begin{subequations}\label{eq:full}
\begin{align}
 u_{jk}^1 =& u_{jk}^0-\Delta tD^x_{ji}u_{i\ell}^0 A_{k\ell}+\Delta tD^{xx}_{ji}u_{i\ell}^0 |A|_{k\ell} + \sigma\Delta t(B_j^1\delta_{k0}-u_{jk}^1)\\
B_j^1 =& B_j^0 + \sigma\Delta t(u_{j0}^1-B_j^1)
\end{align}
\end{subequations}
to the presented energy stable mass conservative DLRA solution from \eqref{eq:schemeStable}. We refer to \eqref{eq:full} as the full system. The total mass at any time $t_n$ shall be defined as $m^n = \Delta x \sum_j \left(u_{j0}^n + B_j^n\right)$. As computational parameters we use $n_x = 1000$ cells in the spatial domain and $N=500$ moments to represent the directional variable. The time step size is chosen as $\Delta t = \text{CFL} \cdot \Delta x$ with a CFL number of $\text{CFL} = 0.99$. In Figure \ref{fig:Planesource} we present computational results for the solution
$f(x,\mu)$, the scalar flux $\Phi = \langle f\rangle_{\mu}$ and the temperature $T$ at the end time $t_{\text{end}}=8$. Further, the evolution of the rank $r$ in time, and the relative mass error $\frac{|m^0-m^n|}{\Vert m^0\Vert}$ are shown. One can observe that the DLRA scheme captures well the behaviour of the full system. For a chosen tolerance of $\vartheta = 10^{-1} \Vert \mathbf{\Sigma} \Vert_2$ the rank increases up to $r=24$ before it reduces again. The relative mass error is of order $\mathcal{O}(10^{-14})$. Hence, our proposed scheme is mass conservative up to machine precision. 

\begin{figure}[h!]
    \centering
    \includegraphics[width = 1.0\linewidth]{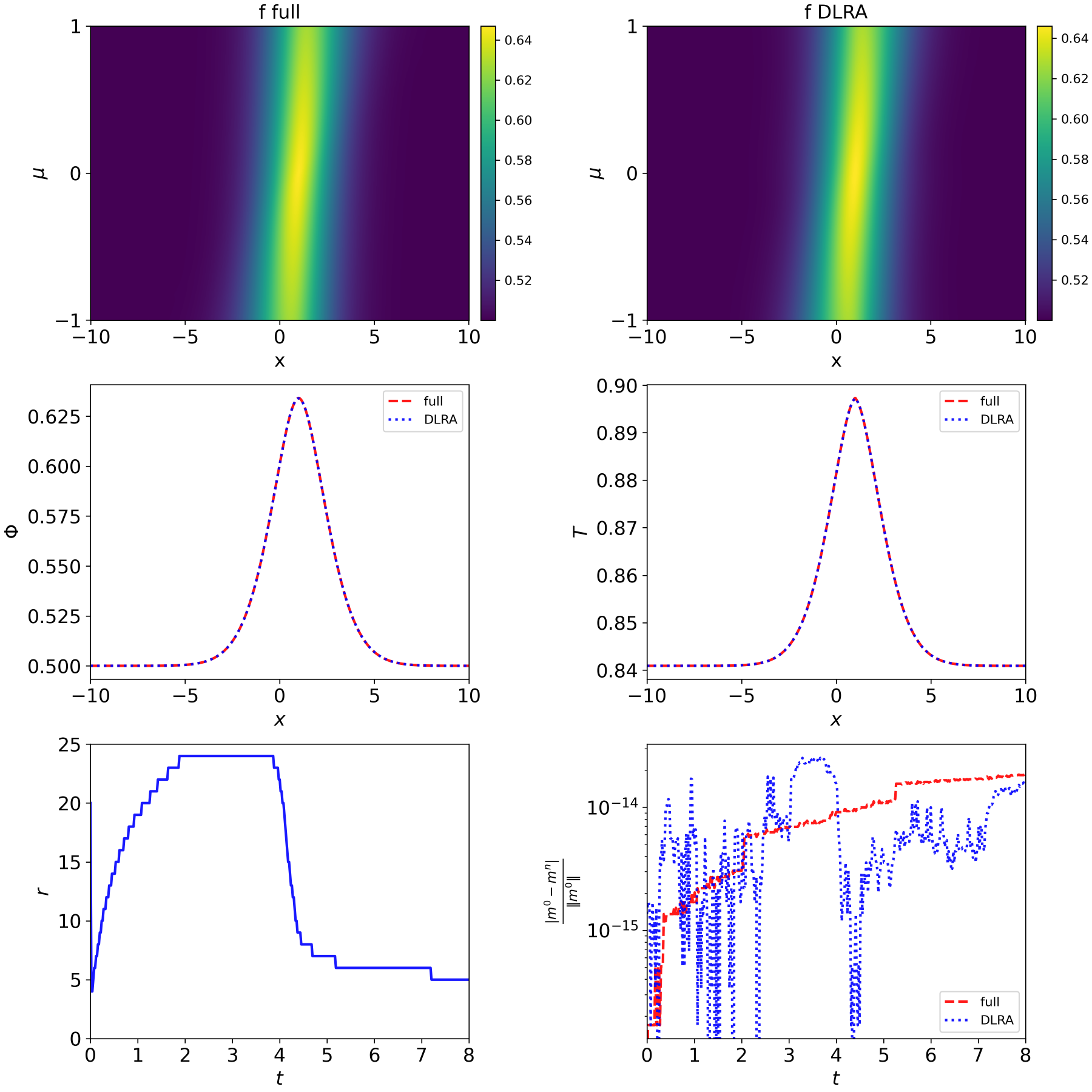}
   \caption{Top row: Numerical results for the solution $f(x,\mu)$ of the plane source problem at time $t_{\text{end}}=8$ computed with the full coupled-implicit system (left) and the DLRA system (right). Middle row: Travelling particle (left) and heat wave (right) for both the full system and the DLRA system. Bottom row: Evolution of the rank in time for the DLRA method (left) and relative mass error compared for both methods (right).}
    \label{fig:Planesource}
\end{figure}

\subsection{1D Su-Olson problem}

For the next test problem we add a source term $Q(x)$ to the previously investigated equations leading to
\begin{subequations}
\begin{align*}
\partial_t f(t,x,\mu) + \mu\partial_x f(t,x,\mu) &= \sigma(B(t,x)-f(t,x,\mu))+ Q(x),\\
\partial_t B(t,x) &= \sigma(\langle f(t,x,\cdot)\rangle_{\mu}-B(t,x)).
\end{align*}
\end{subequations}
In our example we use the source function $Q(x) = \chi_{[-0.5,0.5]}(x)/a$ with $a= \frac{4 \sigma_{\mathrm{SB}}}{c}$ being the radiation constant. Again we 
consider the spatial domain $D = [-10,10]$ and choose the initial condition 
\begin{align*}
u(t=0,x) = \max \left(10^{-4}, \frac{1}{\sqrt{2\pi\sigma_{\mathrm{IC}}^2}} \exp\left(-\frac{(x-1)^2}{2\sigma_{\mathrm{IC}}^2}\right)\right),
\end{align*}
with constant deviation $\sigma_{\mathrm{IC}}=0.03$ and particles moving into all directions $\mu \in [-1,1]$. The initial value for the internal energy is set to $B_0 = 50$, the initial value for the rank to $r=20$. The opacity $\sigma$ is again chosen to have the constant value of $1$. As computational parameters we use $n_x = 1000$ cells in the spatial domain and $N=500$ moments to represent the directional variable. The time step size is chosen as $\Delta t = \text{CFL} \cdot \Delta x$ with a CFL number of $\text{CFL} = 0.99$. The isotropic source term generates radiation particles flying through and interacting with a background material. The interaction is driven by the opacity $\sigma$. In turn, particles heat up the material leading to a travelling temperature front, also called a \textit{Marshak wave} \cite{marshak1958}. Again this travelling heat wave can lead to the emission of new particles from the background material generating a particle wave. At a given time point $t_{\text{end}}= 3.16$ this waves can be seen in Figure \ref{fig:SuOlson} where we display numerical results for the solution $f(x,\mu)$, the scalar flux $\Phi = \langle f\rangle_{\mu}$ and the temperature $T$. We compare the solution of the full coupled-implicit system differing from \eqref{eq:full} by an additional source term to the presented energy stable mass conservative DLRA solution from \eqref{eq:schemeStable} where we have also added this source term. Further, the evolution of the rank in time is presented for a tolerance parameter of $\vartheta = 10^{-2} \Vert \mathbf{\Sigma} \Vert_2$. Again we observe that the proposed DLRA scheme approximates well the behaviour of the full system. In addition, a very low rank is sufficient to obtain accurate results.
Note that due to the source term there is no mass conservation in this example.

\begin{figure}[h!]
    \centering
    \includegraphics[width = 1.0\linewidth]{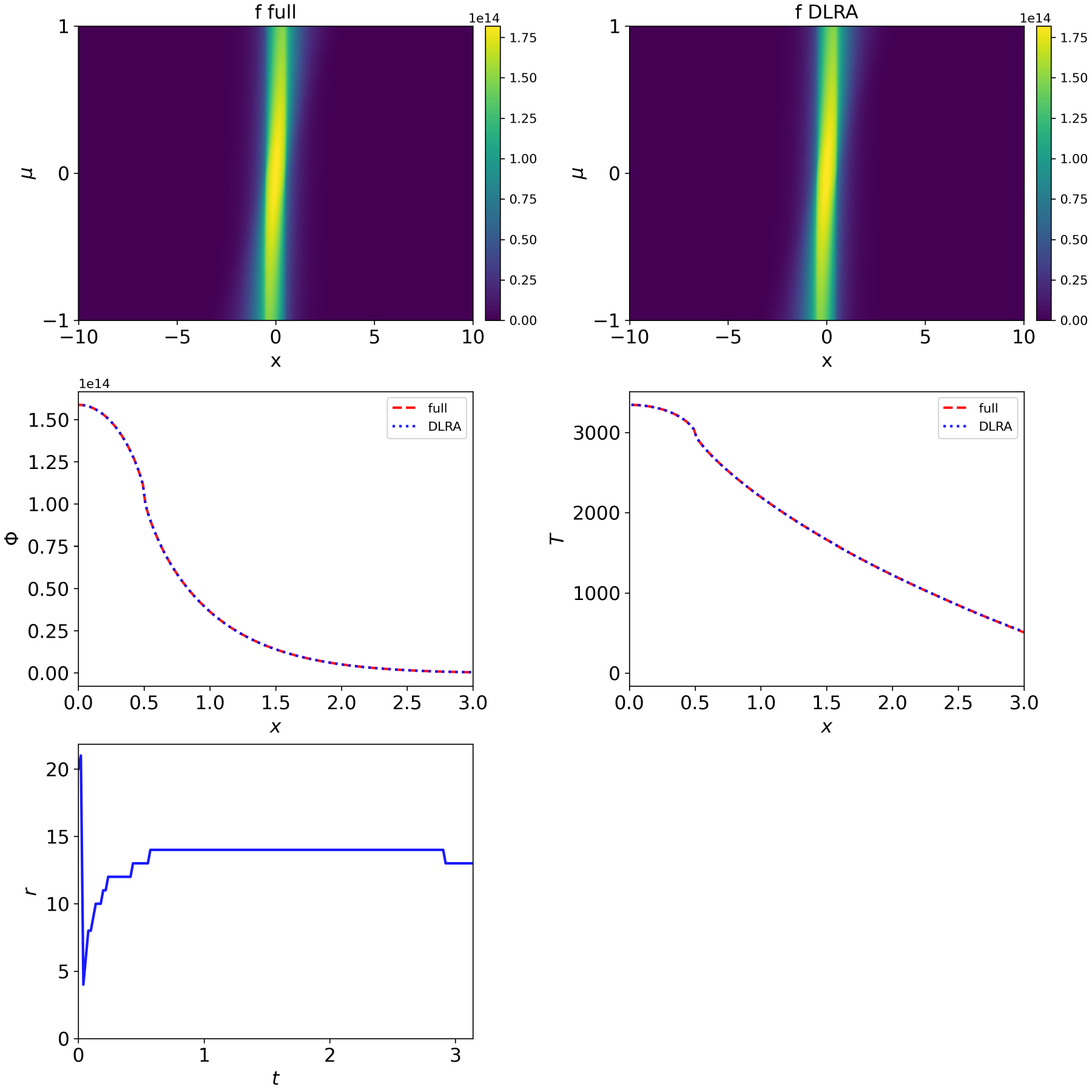}
    \caption{Top row: Numerical results for the solution $f(x,\mu)$ of the Su-Olson problem at time $t_{\text{end}}=3.16$ computed with the full coupled-implicit system (left) and the DLRA system (right). Middle row: Travelling particle (left) and heat wave (right) for both the full system and the DLRA system. Bottom row: Evolution of the rank in time for the DLRA method.}
    \label{fig:SuOlson}
\end{figure}

\subsection{2D Beam}

To approve computational benefits of the presented method we extend it to a two-dimensional setting. The set of equations becomes:
\begin{align*}
\partial_t f(t,\mathbf{x},\mathbf{\Omega}) + \mathbf{\Omega} \cdot \nabla_\mathbf{x} f(t,\mathbf{x},\mathbf{\Omega}) &= \sigma(B(t,\mathbf{x})-f(t,\mathbf{x},\mathbf{\Omega})),\\
\partial_t B(t,\mathbf{x}) &= \sigma(\langle f(t,\mathbf{x},\cdot)\rangle_{\mathbf{\Omega}}-B(t,\mathbf{x})).
\end{align*}
For the numerical experiments let $\mathbf{x}=(x_1,x_2) \in [-1,1] \times [-1, 1], \mathbf{\Omega} =(\Omega_1, \Omega_2, \Omega_3) \in \mathcal{S}^2$ and $\sigma=0.5$. The initial condition of the two-dimensional beam is given by
\begin{align*}
    f(t=0,\mathbf{x},\mathbf{\Omega}) = 10^6\cdot \frac{1}{2\pi\sigma_x^2}\mathrm{exp}\left(-\frac{\Vert \mathbf{x} \Vert^2}{2\sigma_x^2}\right) \cdot \frac{1}{2\pi\sigma_{\Omega}^2}\mathrm{exp}\left(-\frac{(\Omega_1 - \Omega^{\star} )^2 + (\Omega_3 - \Omega^{\star} )^2}{2\sigma_{\Omega}^2}\right),
\end{align*}
with $\Omega^{\star} = \frac1{\sqrt{2}}$, $\sigma_x = \sigma_{\Omega} = 0.1$. The initial value for the internal energy is set to $B^0=1$, the initial value for the rank to $r=100$. The total mass at any time $t_n$ shall be defined as $m^n = \Delta x_1 \Delta x_2 \sum_j \left(u_{j0}^n + B_j^n \right)$. We perform our computations on a spatial grid with $N_{\text{CellsX}}=500$ points in $x_1$ and $N_{\text{CellsY}}=500$ points in $x_2$. For the angular basis we use again a modal approach, namely the spherical harmonics ($P_N$) method. Technical details can be found in \cite{casezweifel1967, mcclarren2008-2, mcclarrenbrunner2010}, whereas \cite{peng2020-2D, kusch2021robust} relates the method to dynamical low-rank approximation. The polynomial degree shall be chosen large enough such that the behaviour is captured correctly but small enough to stay in a reasonable computational regime. An increasing order of unknowns usually leads to an increasing complexity and therefore to the need of a higher polynomial degree. For our example we use a polynomial degree of $n_{\text{PN}}=29$ corresponding to $900$ expansion coefficients in angle. The time step size is chosen as $\Delta t = \text{CFL} \cdot \Delta x$ with a CFL number of $\text{CFL}=0.7$.
We compare the solution of the two-dimensional full system corresponding to \eqref{eq:full} to the two-dimensional DLRA solution corresponding to \eqref{eq:schemeStable}. The extension to two dimensions is straightforward. In Figure \ref{fig:2DBeam-1} we show numerical results for the scalar flux $\Phi = \int_{\mathcal{S}^2} f(t,\mathbf{x}, \cdot) \, \mathrm{d} \mathbf{\Omega}$ and the temperature $T$ at the time $t=0.5$. We again observe the accuracy of the proposed DLRA scheme. For this setup the computational benefit of the DLRA method is significant as the run time compared to the solution of the full problem is reduced by a factor of approximately $8$ from $20023$ seconds to $2509$ seconds. For the evolution of the rank $r$ in time and the relative mass error $\frac{|m^0-m^n|}{\Vert m^0\Vert}$ we consider a time interval up to $t=1.5$. In Figure \ref{fig:2DBeam-2} one can observe that for a chosen tolerance parameter of $\vartheta = 5 \cdot 10^{-4} \Vert \mathbf{\Sigma} \Vert_2$ the rank increases but does not approach its allowed maximal value of $100$. Further, the relative mass error stagnates and the DLRA method shows its mass conservation property.

\begin{figure}[h!]
    \centering
    \includegraphics[width = 0.9\linewidth]{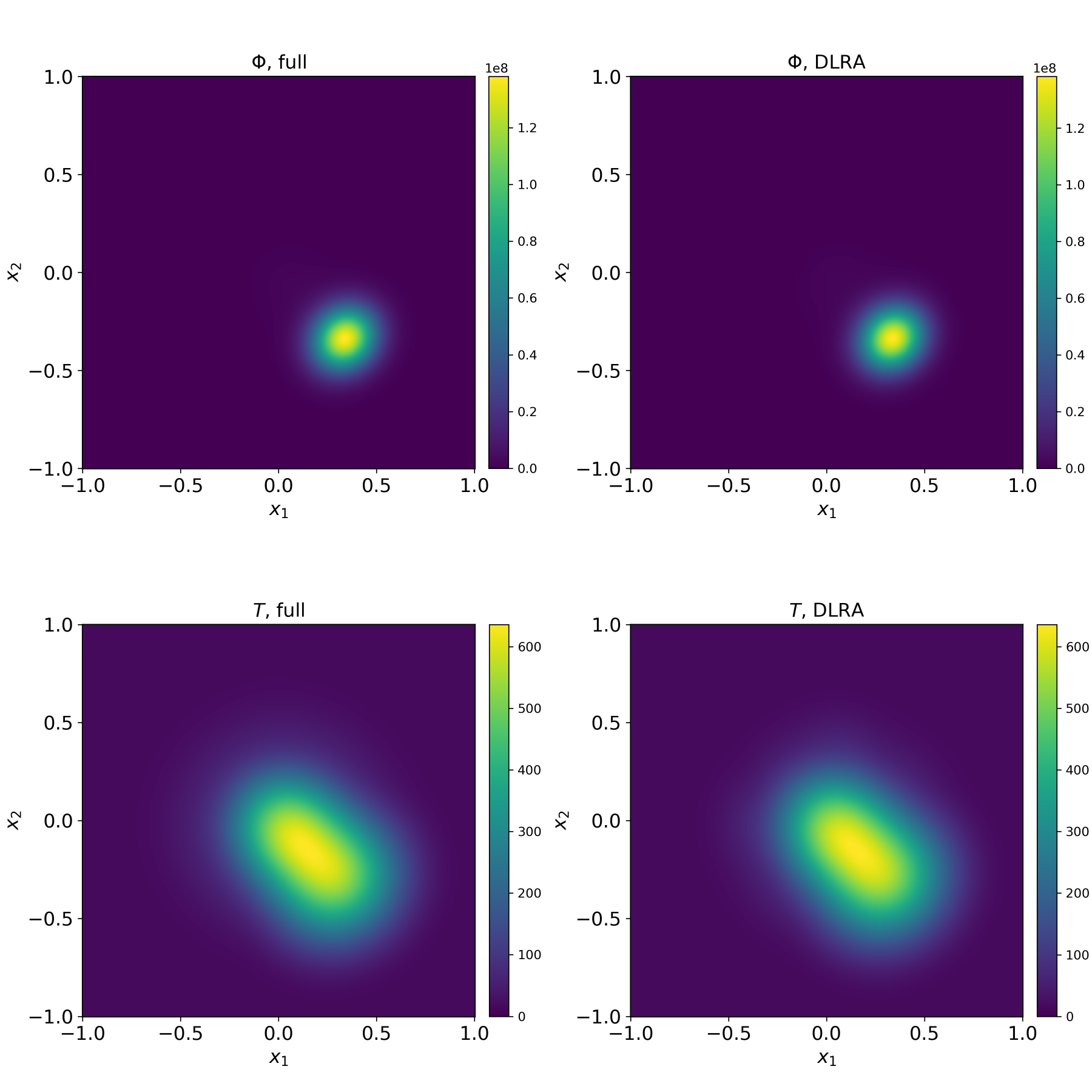}
    \caption{Numerical results of the scalar flux and the temperature for the 2D beam example for the full coupled-implicit system (left) and the DLRA system (right) at the time $t=0.5$. }
    \label{fig:2DBeam-1}
\end{figure}

\begin{figure}[h!]
    \centering
    \includegraphics[width = 0.95\linewidth]{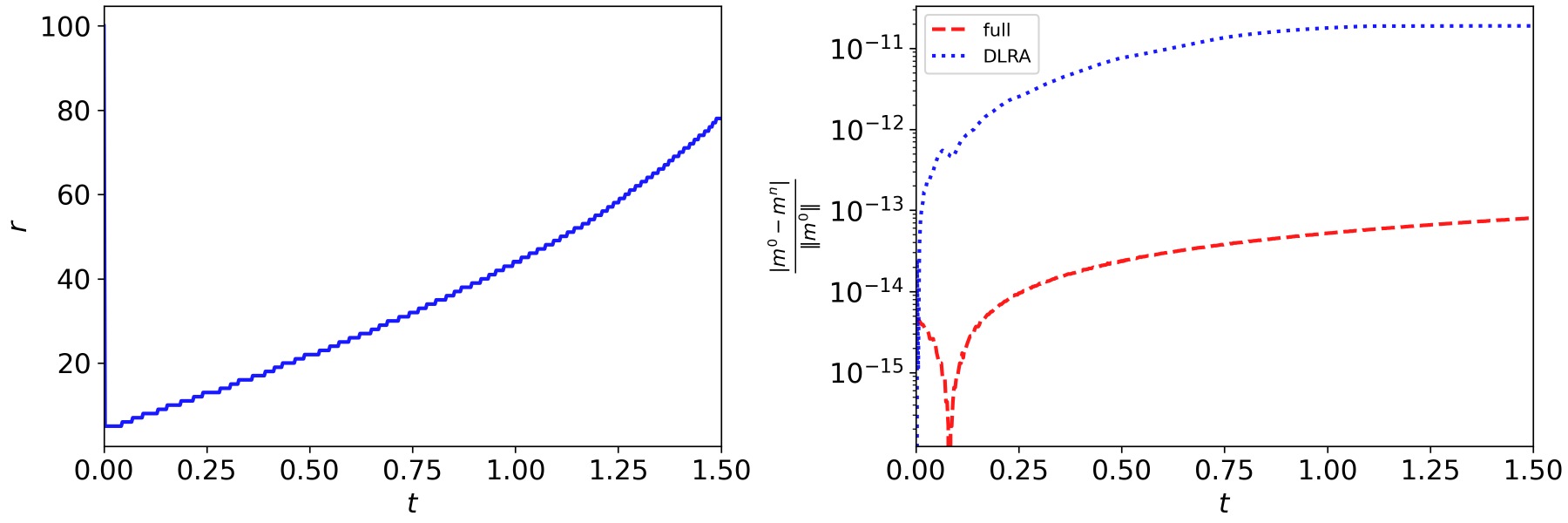}
    \caption{Evolution of the rank in time for the 2D beam example for the DLRA method (left) and relative mass error compared for both methods (right) until a  time of $t=1.5$.}
    \label{fig:2DBeam-2}
\end{figure}

\section{Conclusion and outlook}

We have introduced an energy stable and mass conservative dynamical low-rank algorithm for the Su-Olson problem. The key points leading to these properties consist in treating both equations in a coupled-implicit way and using a mass conservatice truncation strategy. Numerical examples both in 1D and 2D validate the accuracy of the DLRA method. Its efficiency compared to the solution of the full system can especially be seen in the two-dimensional setting. For future work, we propose to implement the parallel integrator of \cite{ceruti2023parallel} for further enhancing the efficiency of the DLRA method. Moreover, we expect to draw conclusions from this Su-Olson system to the Boltzmann-BGK system and the DLRA algorithm presented in \cite{einkemmer2021bgk} regarding stability and an appropriate choice of the size of the time step.\\ \vspace{
10cm}

\section*{Acknowledgements}

Lena Baumann acknowledges support by the Würzburg Mathematics Center for Communication and Interaction (WMCCI) as well as the Stiftung der Deutschen Wirtschaft. The work of Jonas Kusch was funded by the Deutsche Forschungsgemeinschaft (DFG, German Research Foundation) – 491976834.

\section*{Author Contribution Statement (CRediT)}
\vspace{5pt}
{\small
\noindent
\begin{tabular}{@{}lp{10.8cm}}
\textbf{Lena Baumann:} & analysis of energy stability, conceptualization, implementation, plotting, revision, simulation of numerical tests, validation, visualization, writing - original draft \\
\textbf{Lukas Einkemmer:} & analysis of energy stability, conceptualization, initial idea of numerical scheme, proofreading and corrections, supervision\\
\textbf{Christian Klingenberg:} & conceptualization, proofreading and corrections, supervision\\
\textbf{Jonas Kusch:} & analysis of energy stability, conceptualization, implementation, initial idea of numerical scheme, simulation/setup of numerical tests, supervision, visualization, writing - original draft\\
\end{tabular}
}

\newpage
\bibliographystyle{abbrv}
\bibliography{main} 
\end{document}